\documentclass[11pt]{article}
\usepackage{mathrsfs}
\usepackage{mathtools}
\usepackage{amssymb}
\usepackage{amsmath}
\usepackage{amsthm}
\usepackage{amscd}
\usepackage{commath}
\usepackage{enumerate}
\usepackage{stmaryrd}
\usepackage{enumitem}
\usepackage{hyperref}
\usepackage{thmtools}
\usepackage{listings}
\usepackage{tikz}
\usepackage{tikz-cd}
\usepackage{xcolor}
\usepackage{cleveref}

\usepackage{comment}
\usepackage{lipsum}
\usepackage{appendix}
\usetikzlibrary{matrix,arrows,decorations.pathmorphing}
\usepackage{bm}

\DeclareMathOperator{\Vect}{\textup{Vect}}
\DeclareMathOperator{\mO}{\mathcal{O}}
\DeclareMathOperator{\mE}{\mathcal{E}}

\usepackage{relsize} 
\usepackage[bbgreekl]{mathbbol} 
\DeclareSymbolFontAlphabet{\mathbb}{AMSb} 
\DeclareSymbolFontAlphabet{\mathbbl}{bbold} 
\newcommand{\Prism}{{\mathlarger{\mathbbl{\Delta}}}}

\newtheorem{theorem}{Theorem}[section]
\newtheorem{lemma}[theorem]{Lemma}
\newtheorem{proposition}[theorem]{Proposition}

\newtheorem{corollary}[theorem]{Corollary}
\newtheorem*{theorem*}{Theorem}
\newtheorem*{proposition*}{Proposition}
\newtheorem*{corollary*}{Corollary}

\theoremstyle{definition}
\newtheorem{definition}[theorem]{Definition}

\newtheorem{example}[theorem]{Example}
\newtheorem*{claim}{Claim}

\theoremstyle{remark}
\newtheorem{remark}[theorem]{Remark}

\DeclareMathOperator{\Spf}{Spf}

\newcommand{\D}{\mathbb{D}}

\newcommand{\N}{\mathbb{N}}

\newcommand{\Q}{\mathbb{Q}}
\newcommand{\Z}{\mathbb{Z}}
\newcommand{\PP}{\mathbb{P}}
\renewcommand{\P}{\mathbb{P}}

\newcommand{\C}{\mathbb{C}}
\newcommand{\F}{\mathbb{F}}
\newcommand{\V}{\mathbb{V}}
\renewcommand{\L}{\mathbb{L}}
\newcommand{\bL}{\mathbb{L}}

\newcommand{\Fpbar}{\overline{\mathbb{F}}_p}

\newcommand{\ShimK}{\mathrm{Sh}_\mathsf{K}(G,\mathbf{X})}

\newcommand{\integralShimK}{\mathscr{S}_\mathsf{K}(G,\mathbf{X})}

\newcommand{\integralY}{\mathcal{Y}}
\newcommand{\integralS}{\mathscr{S}}

\newcommand{\cO}{\mathcal{O}}

\newcommand{\cF}{\mathcal{F}}
\newcommand{\cH}{\mathcal{H}}
\newcommand{\cS}{\mathcal{S}}
\newcommand{\cP}{\mathcal{P}}
\newcommand{\cZ}{\mathcal{Z}}
\newcommand{\cT}{\mathcal{T}}

\newcommand{\A}{\mathbb{A}}

\DeclareMathOperator{\isom}{\;\xrightarrow{\: {}_{\sim} \:} \;}

\DeclareMathOperator{\Spec}{\mathrm{Spec}}

\DeclareMathOperator{\dR}{dR}

\newcommand{\Spa}{\mathrm{Spa}}

\newcommand{\an}{\mathrm{an}}

\newcommand{\Gm}{\mathbb{G}_m}

\newcommand{\gal}{\mathrm{Gal}}

\newcommand{\Ag}{\mathcal{A}_g}

\newcommand{\AgK}{\mathcal{A}_{g,\mathsf{K}}}

\newcommand{\Ann}{\mathsf{A}}

\newcommand{\scrX}{\mathscr{X}}
\newcommand{\scrD}{\mathscr{D}}
\newcommand{\scrT}{\mathscr{T}}

\newcommand{\Dstar}{\mathsf{D}^{\times}}
\newcommand{\sD}{\mathsf{D}}

\newcommand{\hol}{\mathrm{hol}}

\newcommand{\Ddrlog}{D_\mathrm{dR, log}}
\newcommand{\VetQl}{V_{\text{\'et},\ell}}
\newcommand{\VetZl}{\mathbb{V}_{\text{\'et},\ell}}

\newcommand{\VetZp}{\mathbb{V}_{\text{\'et},p}}
\newcommand{\Vcris}{\mathbb{V}_\mathrm{cris}}
\newcommand{\VFL}{\mathbb{V}_\mathrm{FL}}

\newcommand{\UetZl}{\mathbb{U}_{\text{\'et},\ell}}

\newcommand{\UetZp}{\mathbb{U}_{\text{\'et},p}}
\newcommand{\Ucris}{\mathbb{U}_\mathrm{cris}}
\newcommand{\UFL}{\mathbb{U}_\mathrm{FL}}

\newcommand{\WetZl}{\mathbb{W}_{\text{\'et},\ell}}

\newcommand{\fA}{\mathfrak{A}}

\newcommand{\fD}{\mathfrak{D}}
\newcommand{\fS}{\mathfrak{S}}
\newcommand{\fT}{\mathfrak{T}}
\newcommand{\fX}{\mathfrak{X}}

\newcommand{\LOG}{\mathrm{log}}

\renewcommand{\Im}{\operatorname{Im}}

\DeclareMathOperator{\Hom}{Hom}

\DeclareMathOperator{\Fil}{Fil}
\DeclareMathOperator{\pt}{pt}

\newcommand{\rig}{\mathrm{rig}}

\DeclareMathOperator{\Vdr}{(\mathcal{V},\nabla)_{\textrm{dR}}}
\DeclareMathOperator{\Udr}{(\mathcal{U},\nabla)_{\textrm{dR}}}

\newcommand{\BB}{\mathrm{BB}}

\newcommand{\cris}{\mathrm{cris}}
\newcommand{\CRIS}{\mathrm{CRIS}}
\newcommand{\RH}{\mathcal{R}\mathcal{H}}

\DeclareMathOperator{\spec}{\mathrm{Spec}}

\DeclareMathOperator{\Proj}{Proj}

\DeclareMathOperator{\good}{good}

\def\Spec{\operatorname{Spec}}

\def\sExt{\cE\mathrm{xt}\,}
\def\sHom{\cH\mathrm{om}\,}
\def\gr{\operatorname{gr}}

 \newenvironment{itemize*}
  {\begin{itemize}[topsep=-\parskip+\jot,itemsep=-\parskip-\jot]}
  {\end{itemize}}

\newenvironment{enumerate*}
  {\begin{enumerate}[label=(\alph*),topsep=-\parskip+\jot,itemsep=-\parskip-\jot]}
  {\end{enumerate}}

\newenvironment{enumerate**}
  {\begin{enumerate}[label=(\roman*),topsep=-\parskip+\jot,itemsep=-\parskip-\jot]}
  {\end{enumerate}}
  
\title{$p$-adic hyperbolicity for Shimura varieties and period images}
\author{Benjamin Bakker, Abhishek Oswal, Ananth N. Shankar and Zijian Yao} 
\date{\today}

\def\cE{\mathcal{E}}
\def\Ext{\operatorname{Ext}}
\def\img{\operatorname{img}}

\def\Spec{\operatorname{Spec}}

\def\sExt{\cE\mathrm{xt}\,}
\def\sHom{\cH\mathrm{om}\,}

\def\fX{\mathfrak{X}}
\def\ff{\mathfrak{f}}
\def\fU{\mathfrak{U}}

\def\Fil{\mathrm{Fil}}
\def\scrS{\mathscr{S}}
\def\Hdg{\mathrm{Hdg}}

\begin{document}
\setlist[description]{font=\normalfont\itshape\textbullet\space}
\maketitle

 \begin{abstract}
We prove that Shimura varieties and geometric period images satisfy a $p$-adic extension property for large enough primes $p$. More precisely, let $\Dstar\subset \sD$ denote the inclusion of the closed punctured unit disc in the closed unit disc. Let $X$ be either a Shimura variety or a geometric period image with torsion-free level structure. Let $F$ be a discretely valued $p$-adic field containing the number field of definition of $X$, where $p$ is a large enough prime.  Then, consider a rigid-analytic map $f: (\Dstar)^a \times \sD^b \rightarrow X_F^{\an}$ defined over $F$. If $X$ is a Shimura variety, we show that $f$ extends to a map $(\Dstar)^a \times \sD^b \rightarrow X^{\BB,\an}$, where $X^{\BB}$ denotes the Baily-Borel compactification of $X$. If $X$ is a geometric period image, we show that $f$ extends to a map $\sD^{a+b}\rightarrow X_F^{\an}$ under a good reduction hypothesis --- namely when the image of $f$ intersects the good reduction locus with respect to the integral canonical model. We note that this hypothesis is vacuous if $X$ is proper. We also deduce an application to algebraicity of rigid-analytic maps. Our methods also apply to the more general situation of the rigid generic fiber of formal schemes admitting Fontaine-Laffaile modules which satisfy certain positivity conditions.
 \end{abstract}

\section{Introduction} 
The purpose of this paper is to prove $p$-adic extension and algebraicity theorems for exceptional Shimura varieties and geometric period images. This result is a $p$-adic analogue of the following theorems for complex Shimura varieties that Borel (\cite{borel}) proved in 1972: 
\begin{theorem*}[Borel extension]
    Let $\ShimK$ be a Shimura variety with torsion-free level structure. Let $D$ be the complex open disc and let $D^{\times}$ be the punctured open unit disc. Then, every holomorphic map ${D^{\times}}^a \times D^b \rightarrow \ShimK^{\hol}$ extends to a map $D^{a+b}\rightarrow (\ShimK^{\BB})^{\hol}$. 
\end{theorem*}
An immediate corollary of this extension result and GAGA is the following algebraicity theorem. 

\begin{theorem*}[Borel algebraicity]
    Let $\ShimK$ be as above, and let $M$ be a complex algebraic variety. Then every holomorphic map $M^\hol \rightarrow \ShimK^{\hol}$ is the analytification of an algebraic map $M\rightarrow \ShimK$.
\end{theorem*}

Here is the main theorem of this paper. 

\begin{theorem}\label{thm:introextension}
    Let $X$ be either a Shimura variety or a geometric period image with torsion-free level structure. There exists an integer $N$ with the following property. Let $p$ be a prime that doesn't divide $N$ and suppose $F$ is a discretely valued $p$-adic field containing the field of definition of $X$. Suppose that $f: (\Dstar)^a \times \sD^b \rightarrow X_F^{\an}$ is a rigid-analytic map defined over $F$. 
    \begin{enumerate}
        \item Suppose $X$ is a Shimura variety. Then $f$ extends to a map $\sD^{a+b}\rightarrow X^{\BB,\an}$.

        \item Suppose $X$ is a geometric period image. Further suppose that $\Im(f)$ intersects the good reduction locus of $X_F$. Then $f$ extends to a map $\sD^{a+b} \rightarrow X_F^\an$.
    \end{enumerate}
    
\end{theorem}

Theorem \ref{thm:introextension} has the following corollary (using \cite[\S2.2]{p-adic-borel-Ag}). 
\begin{theorem}\label{thm:introalgebraicity}
    Let $X,p$ and $F$ be as above, and let $M$ be an algebraic variety defined over $F$. Let $f: M^{\an} \rightarrow X^\an$ defined over $F$ be a rigid analytic map. 
    \begin{enumerate}
        \item If $X$ is a Shimura variety, then $f$ is the analytification of an algebraic map. 

        \item If $X$ is a geometric period image, further assume that $\Im(f)$ is contained in the good reduction locus, then $f$ is the analytification of an algebraic map.
        
    \end{enumerate}
    
\end{theorem}

\begin{remark}
    \begin{enumerate}
        \item We define the good reduction locus precisely in Section \ref{Sec: bdry}. Informally, the good reduction locus is the analytic open subspace of $X^{\an}$ whose classical $K$-points arise as $\cO_K$-points of a good integral model of $X$. If $X$ is proper, then the good reduction locus is all of $X^{\an}$ and therefore the good reduction hypothesis is vacuous. We expect the theorem to hold without this hypothesis. 

        \item We draw the reader's attention to the fact that the good reduction hypothesis in the period image case of Theorem \ref{thm:introextension} has the consequence that the extension of $f$ yields a map to $X$, and it is not necessary to compactify $X$. 

    \end{enumerate}
\end{remark}

\subsection{Other results}
The proofs of Theorem \ref{thm:introextension} and \ref{thm:introalgebraicity} work in a more general setting than just the case of geometric period images and Shimura varieties. In order to not mire ourselves in unenlightening notation and technicalities, we will state a result that is not the most general but that is the cleanest to state. 

\begin{theorem}\label{thm: Brunebarbe}
    Let $\mathscr{X}$ be a smooth scheme over $W(\Fpbar)$, and let $\mathscr{X}^{\rig}$ denote its rigid generic fiber. Let $\bL/\mathscr{X}^{\rig}$ be a crystalline local system with $\VFL$ the associated Fontaine-Laffaile module. Suppose that we are in one of the following two cases. 
    \begin{enumerate}
        \item The Kodaira-Spencer map associated to the filtered flat bundle underlying $\VFL$ is everywhere immersive. 
        \item The Griffiths bundle associated to $\VFL$ is an ample bundle on $\mathscr{X}$. 
    \end{enumerate}
    Then, every map $(\Dstar)^{a} \times \sD^b \rightarrow \mathscr{X}^{\rig}$ extends to a map $\sD^{a+b}\rightarrow \mathscr{X}^{\rig}$ 
\end{theorem}

\subsection{Outline of proof}

The main results of \cite{p-adic-borel-Ag} proved the $p$-adic extension and algebraization results for Shimura varieties \emph{of abelian type}  parallelling Borel's theorem in the setting of a discretely valued $p$-adic field. The strategy of \cite{p-adic-borel-Ag} crucially uses the existence of Rapoport-Zink (\cite{RZ}) spaces and Rapoport-Zink uniformizations of $\Ag$. This in turn of course relies on the moduli interpretation of $\Ag$. While there is a theory of Rapoport-Zink spaces (see \cite{RV}) that goes beyond the setting of abelian varieties, it is not known (though it is certainly expected) that exceptional Shimura varieties admit such uniformization maps. The setting of geometric period images is even more barren, without even any expectations of $p$-adic uniformization maps. Our proof therefore sidesteps the existence Rapoport-Zink uniformizations and instead make strong use of the existence of crystalline local systems and Fontaine-Laffaile modules. 

The outline of our proof is as follows.  The main step is the case of a one-dimensional disk. For brevity, we will focus on the Shimura case. We work at a prime $p$ at which $X$ has an integral canonical model (which we will denote by $\mathscr{X}$ in the introduction). We have a log-smooth compactification $\bar{\scrX}$ of $\scrX$, and $\mathscr{X}$ is equipped with $\ell$-adic local systems and a log Fontaine-Laffaille module, and $X$ is equipped with a crystalline $p$-adic local system associated to the log Fontaine-Laffaille module. We first prove that any map $f: \Dstar \rightarrow X_F^{\an}$ has the property that $f(\Dstar)$ is either entirely contained in the good reduction locus, or the bad reduction locus. This proof is $\ell$-adic and follows the arguments in \cite{p-adic-borel-Ag}, and uses a monodromy-theoretic description of the good-reduction locus proved in \cite{PST} for exceptional Shimura varieties. In fact, smilar to \cite{p-adic-borel-Ag}, we prove that the every point in the image of $\Dstar$ has the same reduction type, i.e. all the points specialize to a single open boundary stratum of $\scrX^{\BB}$. Let us first focus on the good-reduction case. We prove that the $p$-adic local system extends to $\sD$ (recent work of \cite{DY} proves a more general extension result, but in our case we are able to prove this by using the $p$-adic Riemann-Hilbert correspondence as proved by \cite{DLLZ}). We then apply the theory of prismatic $F$-crystals to show that up to shrinking $\sD$, the $F$-crystal associated to the crystalline Galois representation $\bL_x$ is independent of the classical point $x\in \sD$. Now, we write $\Dstar$ as an increasing union of annuli $\Ann_k$, each of which admits an integral model $\mathfrak{A}_k$ that maps to $\mathscr{X}$. We then generalize an argument of Oort (\cite{oortleaves}) to show that the $F$-crystal on $\mathscr{X}\bmod p$ pulls back to $\mathfrak{A}_k \bmod p$. Finally, we use the Kodaira-Spencer map to prove that any map from a connected variety over $\Fpbar$ to $\mathscr{X}\bmod p$ with the property that the $F$-crystal over $\mathscr{X}\bmod p$ pulls back to something constant must in fact be the constant map. We then conclude that the map from $\Dstar$ to $X_F^\an$ maps to a residue disc, and therefore extends by the Riemann extension theorem. 

In the setting of bad-reduction, we work with a fixed Baily-Borel boundary stratum of $\scrT \subset \bar{\mathscr{X}}$ where $\scrT$ is  the fiber of the map $\bar{\mathscr{X}} \rightarrow \scrX^{\BB}$ over an open boundary stratum. Note that such an open boundary stratum of $\mathscr{X}^{\BB}$ is a smaller-dimensional Shimura variety in its own right. We prove that the log $F$-crystal (underlying the log FL module) restricted to $\scrT_{\overline{\F}_p}$ admits a weight-filtration by sub log $F$-crystals, and that the associated graded is an $F$-crystal pulled-back from the open boundary stratum of $\scrX^{\BB}$. We then use this weight filtration together with work from \cite{DLMS2} to construct a filtration of the $p$-adic local system restricted to $\Dstar$. This allows us to reduce to the good reduction case, and we thus finish the case of maps from $\Dstar$. 

To deduce \Cref{thm:introextension} for morphisms from polydisks $f: (\Dstar)^a \times \sD^b \rightarrow X_F^{\an}$, we show that the existence of an extension on any one-dimensional disk implies that $f$ extends meromorphically and then use the $p$-adic Riemann--Hilbert correspondence of \cite{DLLZ} to show that the exceptional fibers in the resolution of indeterminacies of $f$ must be contracted.  

\subsection{Previous work}
There are several results prior (aside from Borel's work) to our work that addresses the questions of algebraicity and extension -- both in the complex and $p$-adic settings. In the complex case, \cite{bbt-o-minimal-gaga} and \cite{bfmt-bailyborel} prove the algebraicity and extension results  for geometric period images.

As earlier mentioned, \cite{p-adic-borel-Ag} treats the case of abelian Shimura varieties for all primes $p$, without a good-reduction hypothesis. It also treats the case of the universal abelian scheme over compact Shimura varieties of Hodge type, and Rapoport-Zink spaces associated to $\AgK$. The paper \cite{OswalPappas} proves the $p$-adic extension theorem for local Shimura varieties. Cherry (in \cite{Cherry}) addresses the case of genus $\geq 2$ curves in the more general situation of $\C_p$. Cherry-Ru (\cite{cherry-ru}) prove a $p$-adic big Picard style theorem, and Sun (\cite{sun-hyperbolicity-of-mg}) proves the $\C_p$-analogue of the algebraicity theorem.

\subsection{Organization of the paper}
In Section \ref{sec: setup}, we introduce various objects that live on Shimura varieties and period images. In Section \ref{Sec: bdry}, we prove that the image of every map $\Dstar \rightarrow X$ must either be entirely contained in the good reduction locus or the bad reduction locus. In Section \ref{sec: prismatic setup}, we recall results about prismatic $F$-crystals, and in Section \ref{Sec:pointwiseconstancy} prove a crucial constancy result for $F$-crystals. In Section \ref{sec: constancy}, we generalize work of Oort to show that a pointwise constant $F$-crystal on $\mathbb{P}^1$ must be constant. In Section \ref{sec: Dstar}, we prove the main theorem for $\Dstar$ in the good reduction case. In Section \ref{section:bad_reduction_case}, we recall notions about log structures and deal with the bad reduction case for Shimura varieties. Finally, we prove the main theorem in general in Section \ref{sec: higher borel}.
\subsection{Acknowledgements}
We are very grateful to H\'{e}l\`{e}ne Esnault,  Haoyang Guo, Kentaro Inoue, Ruofan Jiang, Teruhisa Koshikawa, Anand Patel, Koji Shimizu, Jacob Tsimerman, Alex Youcis, Ziquan Yang, and Xinwen Zhu for several helpful conversations. We thank Alex Youcis for pointing us to the reference \cite{IKY} and Ruofan Jiang for pointing out to us the reference \cite[Proposition 3.46 (c)]{DLMS2}. We are especially grateful to Xinwen Zhu for suggesting the problem of a $p$-adic Borel extension to the second author in the first place, and to H\'{e}l\`{e}ne Esnault for suggesting to the third author that a similar extension theorem might hold in the presence of a suitable local system. B. B. was partially supported by NSF grant DMS-2401383, the Institute for Advanced Study, and the Charles
Simonyi Endowment. A. S. was partially supported by the NSF grants DMS-2338942 and DMS-2424441, and a Sloan research fellowship. B.B. and A.S. thank the IAS for their hospitality during which part of this paper was written. 
 
\section{Notations and the general setup}\label{sec: setup}

\subsection{The set-up for the Shimura case}
We follow closely the notations of \cite{bst-integral-canonical}.
To recall, $(G,\mathbf{X})$ shall denote a Shimura datum. Let $E := E(G,\mathbf{X}) \subset \C$ denote the reflex field.
We let $V$ be the adjoint representation of $G$, a $\Z$-lattice $\mathbb{V}$ of $V$, and a neat compact open subgroup $\mathsf{K} \subset G(\A_f)$ that stabilizes $\mathbb{V}\otimes \hat{\Z} \subset V\otimes_\Q \A_f.$ We further assume that $\mathsf{K}$ acts trivially on $\mathbb{V}/3\mathbb{V}.$
Denote by $\ShimK$, the corresponding Shimura variety over $E$.
Associated to $\mathbb{V} \subset V$, we have a family of $\Z_\ell$ (respectively $\Q_\ell$) \'etale local systems on $\ShimK$ that we denote by  $\VetZl$ (resp. $\VetQl$.)   
We denote by $V_\mathrm{dR} := (\mathcal{V},\nabla,F^\bullet \mathcal{V})$ the associated filtered flat vector bundle on $\ShimK$ defined over $E$.  

We pick a large integer $N$ as in \cite[Theorem 1.3]{bst-integral-canonical}, so that $\ShimK$ admits a smooth model $\integralShimK$ over $\cO_E[1/N]$ and such that for all places $v$ of $E$ outside $N$, $\integralShimK \otimes_{\cO_E[1/N]} \cO_{E_v}$ is a \emph{canonical} integral model of $\ShimK \otimes_E E_v$ over the ring of integers $\cO_{E_v}$ of the $v$-adic completion $E_v$ of $E$. Furthermore, for every place $v \nmid N$, $\integralShimK \otimes_{\cO_E[1/N]} \cO_{E_v}$ admits a log-smooth compactification over $\cO_{E_v}$. The $\Z_\ell$-\'etale local systems $\VetZl$ on $\ShimK \otimes_E E_v$, extend to $\ell$-adic \'etale local systems on $\integralShimK \otimes_{\cO_E[1/N]} \cO_{E_v}[1/\ell]$. We also have that for $p \nmid N$, the restriction of $\VetZp$ to $\ShimK \otimes_E E_v$ is crystalline in the sense of Faltings-Fontaine-Laffaille, where $v$ is a place of $E$ dividing $p$. By increasing $N$ if necessary, we may also assume that $\Vdr$ spreads out to $\integralShimK$ such that the Kodaira-Spencer map is everywhere immersive. 

We shall fix henceforth a rational prime $p \nmid N$, a place $v$ of $E$ above $p$. 
We shall work throughout over the $p$-adic local field $K := E_v$.
Set $\mathscr{S} := \integralShimK \otimes_{\cO_E[1/N]} \cO_{E_v}$, and $S := \ShimK \otimes_E E_v.$
By a slight abuse of notation, we denote the pullbacks to $S$ of the local systems $\VetQl$, $\VetZl$ and the filtered flat vector bundle $V_\mathrm{dR}$ also by $\VetQl$, $\VetZl$ and $V_\mathrm{dR}$ respectively, and in the case $\ell \neq p$, their extensions to the integral canonical model $\mathscr{S}$ shall also be denoted by the same. We let $\VFL$ denote the Fontaine-Laffaile module associated to $\VetZp/S$ on the formal $p$-adic completion $\hat \integralS$, and we let $\V_{\cris}$ denote the $F$-crystal on $\integralS_p := \integralS \otimes_{\cO_{E_v}} k_v$, where $k_v$ is the residue field of $\cO_{E_v}$. Note that the filtered flat bundle underlying $\VFL$ is just $\Vdr$. We will sometimes use the symbol $\bL$ to denote the $\Z_p$-\'etale local system $\VetZp$ on $S$.

We denote by $\ShimK^\BB$ the Baily--Borel compactification of $\ShimK$ and set $S^\BB := \ShimK^\BB \otimes_E E_v$.  $\ShimK^\BB$ is naturally stratified by Shimura varieties, and we may assume this spreads out to a stratification of an integral model $\scrS^\BB$ of $S^\BB$ over $\cO_{E_v}$ by integral models of those Shimura varieties in the above sense.  We may further assume there is:
\begin{enumerate}
    \item A log smooth compactification\footnote{Specifically meaning $(\bar\scrS,\scrD)$ is \'etale-locally isomorphic to $(\A^m_{\cO_{E_v}},\mathscr{C})$ where 
    $\mathscr{C}$ is a union of coordinate hyperplanes.} $\bar\scrS$ of $\scrS$ over $\cO_{E_v}$ with boundary $\scrD:=\bar\scrS\backslash \scrS$ equipped with a morphism $\bar\scrS\to\scrS^\BB$ extending the identity on $\scrS$.  We refer to the inverse images of strata under this morphism as Baily--Borel strata of $\bar\scrS$, and typically denote them by $\scrT$.  We may after modification of $\bar\scrS$ (by blowing up components of intersections of divisors) assume the Baily--Borel strata of $\bar\scrS$ are unions of boundary divisors.
    
    \item $\VFL$ extends to a log Fontaine--Laffille module $\bar \V_\mathrm{FL}$ on $\bar \scrS$.  We denote its underlying filtered logarithmic flat vector bundle by $\bar \V_\mathrm{dR}$, and log $F$-crystal by $\bar \V_{\cris}$.  Equipping $\bar \scrS$ with its natural log structure $(\bar \scrS,M_{\bar \scrS})$ of functions which are invertible off $\scrD$, we may interpret $\bar \V_{\mathrm{dR}}$ (resp. $\bar\V_{\cris}$) as a filtered flat vector bundle (resp. $F$-crystal) on $(\bar \scrS,M_{\bar \scrS})$.  
    
    \item For each connected component $\scrD_0$ of an intersection of irreducible components of $\scrD$, $\scrD_0$ is naturally equipped with a log structure $(\scrD_0,M_{\scrD_0})$ making it a strict sub-log scheme.  The restriction of the flat vector bundle $\bar \V_\mathrm{dR}$ to $\scrD_0$ naturally acquires a monodromy-weight filtration $W_\bullet \bar \V_\mathrm{dR}|_{\scrD_0}$ by sub-flat vector bundles coming from the residues of the connection.
    
    \item For each Baily--Borel stratum $\scrT$, these monodromy weight filtrations glue to give a filtration $W_\bullet \bar \V_\mathrm{dR}|_\scrT$ by sub-flat vector bundles. 
\end{enumerate}

\subsection{The set-up for geometric period images}

Let $E \subset \C$ be a number field, $P$ a smooth, connected quasi-projective algebraic variety over $E$, and $f : Z \rightarrow P$ a smooth, projective $E$-morphism.
For a fixed $m$, we have a polarizable integral variation of Hodge structures $(\mathbb{W}_\Z,F^\bullet)$ with underlying $\Z$-local system $\mathbb{W}_\Z := R^mf^\hol_\ast(\underline{\Z}_{Z^\hol})$. 
Denote by $G$, the generic Mumford-Tate group of the variation.
We shall further assume that the variation has \emph{neat monodromy}, which can always be arranged after passing to a finite \'etale cover of $P$.
We denote by $\mathbb{W}_{\text{\'et},\ell} := R^mf^{\text{\'et}}_\ast(\underline{\Z_\ell})$ the associated $\Z_\ell$ local system on $P$, and by $W_\mathrm{dR} := (\mathcal{W},\nabla, F^\bullet)$ the filtered flat algebraic vector bundle on $P$ defined over $E$, such that $\mathbb{W}_\Z$ is the sheaf of flat sections of the associated analytified filtered flat bundle on $P^\hol$.  

We denote by $Y_{/E}$ the Stein factorization of the period map associated to the variation $(\mathbb{W}_\Z,F^\bullet)$, in the sense of \cite[\S 1.4]{bst-integral-canonical} (see also \cite{bbt-o-minimal-gaga}). Thus, 
$Y_{/E}$ is a quasi-projective algebraic variety over $E$, and the
associated period map $\phi : P_\C^\hol \rightarrow \Gamma \backslash D$, factors as a composite $P^\hol \xrightarrow{g^\hol} Y_{\C}^\hol \rightarrow \Gamma \backslash D,$ where the morphism $Y_{\C}^\hol \rightarrow \Gamma \backslash D$ is finite, and $P \xrightarrow{g} Y_{/E}$ is an algebraic map defined over $E$ with geometrically connected generic fiber. There also exists a smooth partial compactification $P'$ of $P$ defined over $E$ and a proper map $P' \rightarrow Y$, with $W_\mathrm{dR}$, $\mathbb{W}_{\text{\'et},\ell}$ (resp. $(\mathbb{W}_\Z,F^\bullet)$) extending to $P'$ (resp. $P^{\prime,\hol}$).

Note that the variation $(\mathbb{W}_\Z,F^\bullet)$ descends to a polarizable $\Z$ variation of Hodge structures $(\mathbb{V}_\Z,F^\bullet)$ on $Y_{\C}^\hol$, as do the $\Z_\ell$-\'etale local systems $\WetZl$ to $\Z_\ell$-\'etale local systems $\VetZl$ on $Y_{/E}$ (see \cite[\S 2.6]{bst-integral-canonical}).
The filtered flat bundle $W_\mathrm{dR} = (\mathcal{W},\nabla,F^\bullet)$ on $P$ descends to a filtered vector bundle on $Y_{/E}$ defined over $E$, denoted by $V_\mathrm{dR} = (\mathcal{V},F^\bullet)$. 

There is a finite set of places $\Sigma$ of $E$ such that:
\begin{itemize}
    \item $Z\rightarrow P$ spreads out to a smooth proper family over a smooth base $\cZ \rightarrow \cP$ over $\cO_{E,\Sigma}$.

    \item $P'\rightarrow Y$ spreads out to a proper map $\cP' \rightarrow \integralY$ over $\cO_{E,\Sigma}$ where $\cP'$ is a smooth partial compactification of $\cP$. 

    \item The filtered flat bundle $W_\mathrm{dR}/P'$ and the filtered bundle $V_{\dR}/Y$ spread out to a filtered flat bundle on $\cP'$ and a filtered bundle on $\integralY$.  We abusively denote the extensions by the same notation. The Griffiths bundle of $V_\mathrm{dR}$ on $\integralY$ is ample. 
    \item For every prime $\ell$, the local system $\WetZl$ extends to an $\ell$-adic local system on $\cP'_{\cO_{E,\Sigma_\ell}}$ where $\Sigma_\ell$ is the union of $\Sigma$ and all primes of $E$ dividing $\ell$.  We abusively denote these local system by $\WetZl$ as well.   Likewise, the $\Z_\ell$-\'etale local systems $\VetZl$ on $Y$ extend to $\Z_\ell$-local systems on $\integralY_{ \cO_{E,\Sigma_\ell}}$, which we denote by the same notation.

    \item $\mathcal{Y}$ is an integral canonical model (as in \cite{bst-integral-canonical}) of $Y$ over $\cO_{E,\Sigma}$. 

    \item There is a uniform stratified resolution with boundary $\cS^j \rightarrow \integralY^j$ of $\integralY$ over $\cO_{E,\Sigma}$ as in \cite[Definition 4.2]{bst-integral-canonical}, and for each $j$, there is a smooth scheme $\cT^j/\cO_{E,\Sigma}$ and maps $\cT^j \xrightarrow{t^j} \cP'$ and $\cT^j \xrightarrow{q^j} \cS^j$ as in \cite[Section 5.2]{bst-integral-canonical}.  For the largest-dimesnional stratum $\mathcal{Y}^m=\mathcal{Y}$, we may assume $\mathcal{P}'\to \mathcal{Y}$ factors through $\mathcal{S}^m$, and that $\mathcal{P}'\to\mathcal{S}^m$ has geometrically connected fibers.

    \item We let $\UetZl^j$ denote the pullback of $\VetZl|_{\integralY^j\times \cO_{E,\Sigma_\ell}}$ to $\cS^j \times \cO_{E,\Sigma_\ell}$, and let $\Udr^j$ be the pullback of $V_{\dR}|_{\integralY^j}$---note that $\Udr^j$ is a filtered flat bundle.
 
\end{itemize}

  The filtered flat bundle $W_{\mathrm{dR}}/\cP'$ has the structure of a Fontaine-Laffaille module at a prime $v\notin \Sigma$ and this corresponds to the local system $\mathbb{W}_{\textrm{et},p}$ for $v\mid p$ via the Faltings-Fontaine-Laffaile correspondence. There are two filtered flat bundles on $\cS^j$---one is $\Udr^j$, already defined. The other one is the filtered flat bundle underlying the Fontaine-Laffaille module associated to $\UetZp^j$. We denote this Fontaine-Laffaille module by $\UFL^j$. We note that both these filtered flat bundles become isomorphic when pulled back to $\cT^j$.  Note however that for the largest stratum $\mathcal{S}^m$, these two filtered flat vector bundles \emph{are} isomorphic, since $\mathcal{P}'\to\mathcal{S}^m$ has geometrically connected fibers.  

Henceforth, we fix once and for all a finite place $v \notin \Sigma$. Let $p$ denote the rational prime below $v$. Set $K := E_v$, with ring of integers $\cO_K$ and residue field $k$. We will let $Y, P, V_{\dR}$ (resp. $\integralY, \cS, \cT, $) etc. also denote the basechange of the corresponding objects from $E$ (resp. $\cO_{E,\Sigma}$) to $K$ (resp. $\cO_K$). Note that the $p$-adic (and $\ell$-adic) local systems on all the spaces agree under pullback. We will sometimes use the symbol $\bL$ to denote our $p$-adic local system(s) if the base is implicit (or unimportant). 
We denote by $Y^\BB$ the Baily--Borel compactification of the period image $Y$ (see \cite{bfmt-bailyborel}).

\subsection{General notations}

As above, we fix a rational prime $p$, a $p$-adic field $K$ with ring of integers $\cO_K$, and residue field $k$.

For a complete non-archimedean field extension $F$ of $K$, rigid-analytic varieties and spaces over $F$ shall be viewed as adic spaces over $\Spa(F,\cO_F)$. In particular, by a rigid-analytic variety over $F$, we shall mean a quasi-separated adic space that is locally of finite type over $\Spa(F,\cO_F).$ 
For an admissible formal $\cO_F$-scheme $\mathscr{X}/\Spf(\cO_F)$, we denote the associated rigid-analytic generic fiber by $\mathscr{X}^{\rig}$.
For an algebraic variety $X$ over $F$ (respectively a morphism $g : W \rightarrow X$ of algebraic varieties over $F$), we denote by $X^\an$ the associated rigid-analytic space over $\Spa(F)$ (resp. by $g^\an : W^\an \rightarrow X^\an$ the associated morphism of rigid-analytic spaces over $F$). 

For a complex algebraic variety $X \rightarrow \Spec(\C)$ (respectively a morphism of complex algebraic varieties $g : W \rightarrow X$) we denote the associated complex analytic space by $X^\hol$ (respectively the associated morphism of complex analytic spaces by $g^\hol$).

The rigid-analytic closed unit disk over $F$ is denoted by $\sD_F := \Spa(F\langle t \rangle, \cO_F\langle t \rangle).$ 
The punctured closed unit disk over $F$ is $\Dstar_F := \sD_F \setminus \{t = 0\}.$

\section{Boundary and interior}\label{Sec: bdry}

\begin{definition}
    Let $X$ denote either the Shimura variety $S$ or the period image $Y$, and let $\mathscr{X}$ denote the integral canonical model. Let $K$ be a discretely valued field with ring of integers $\cO_K$. We say that a point $x \in \mathscr{X}(K)$ has good reduction if its specialization lies in the interior, i.e.  $x$ is induced by an $\cO_K$-valued point of $\mathscr{X}$. We say that $x$ has bad reduction otherwise. Define the \emph{good reduction locus} $X^{\good}$ of $X$ to be the set of points of $X^\an$ whose mod $p$ specialization with respect to the canonical model lies in the interior. We define the bad reduction locus to be the complement in $X^\an$ of the good reduction locus. We note that $X^{\good}$ is an analytic open subspace of $X^\an$.

\end{definition}

We will first prove that a map $f: \Dstar \rightarrow X^\an$ must either be contained entirely in the good reduction locus or the bad reduction locus where $X$ is as above. This argument appears in an old arXiv version of \cite{bst-integral-canonical} (Lemma 6.4) (which is turn is essentially the same as the argument in the abelian case \cite[Theorem 3.3]{p-adic-borel-Ag}) but we include it in this paper for completion. 

\begin{theorem}\label{thm: all good or all bad}
    Let $f : \Dstar \rightarrow X^\an$ be an analytic map where $X$ is as above. Then either $f(\Dstar) \subseteq X^{\good}$ or $f(\Dstar) \subseteq (X^\an \setminus X^{\good})$. 
\end{theorem}

The following lemma is an analogue of the Neron-Ogg-Shafarevich criterion and follows directly by the arguments of \cite[Lemma 8.4]{PST}. 
\begin{lemma}\label{lem: NOS}
Let $K$ either be $\F((t))$ or a discretely valued $p$-adic field, and let $x\in X(K)$ be a point.  Then $x$ has bad reduction if and only if the action of the inertia subgroup $I_K \subset \gal_K$ on $(\VetZl)_{x}$ is quasi-unipotent of infinite order.
\end{lemma}

\begin{proof}[Proof of \autoref{thm: all good or all bad}]
\cite[Proposition 3.6]{p-adic-borel-Ag} does what is required for thin annuli. To deduce the result for thick annuli from thin annuli, we proceed along identical lines to the argument in \cite[Section 3.1.2]{p-adic-borel-Ag}. We are reduced to proving the following result. Let $g: \Gm \rightarrow \mathscr{X}_p$ be any map. Then, $g$ extends to a map $\PP^1 \rightarrow \mathscr{X}_p$. To prove this, it suffices to prove that the local monodromy of $g^{-1}\VetZl$ around the boundary points is semi-simple. To prove this, it suffices to prove that the geometric monodromy of $g^{-1}\VetZl$ is semisimple. By \cite[3.4.12]{Weil2}, it suffices to prove that the arithmetic local system $g^{-1}\VetZl$ is point-wise pure---of course, this would follow from proving that $\VetZl$ itself were pure. 

In the Shimura case, this follows from the fact that $\VetZl$ is induced by the adjoint representation of $G$, and is therefore an irreducible local system with finite determinant. In the case of geometric period maps, this follows from the fact that the local system is induced by a geometric family, and is therefore point-wise pure by \cite{Weil2}.  
\end{proof}

In the setting of Shimura varieties, we will need the following strengthening of Theorem \ref{thm: all good or all bad}, which is a generalization of \cite[Theorem 3.3]{p-adic-borel-Ag} to the setting of non-abelian Shimura varieties. 

We recall for the convenience of the reader that we use the term \emph{Baily-Borel stratum} for the inverse image in $\bar{\scrS}$ of some boundary stratum in $\scrS^{\BB}$. 

\begin{proposition}\label{prop: tube over bb stratum}
   Let $f: \Dstar \rightarrow S^{\an}$ be a map whose image is contained in the bad-reduction locus. Then, there is a Baily-Borel stratum $\scrT$ such that the image of $f$ is contained in the tube $T_{\scrT}$ over $\scrT$.  
\end{proposition}
\begin{proof}
Fix a Baily-Borel stratum $\scrT$. Similar to Lemma \ref{lem: NOS}, we have that the rank of inertial invariants of $\VetZl|_x$ is independent of the point $x\in T_{\scrT} \setminus \scrT_{E_v}^{\an}$. We will refer to this number as the \emph{inertial rank} of $T_{\scrT}$. Unlike the case of $\Ag$, it isn't a-priori true that the rank of inertial invariants determines the Baily-Borel stratum we work with. However, given two open boundary strata $\scrS_1, \scrS_2 \subset \scrS^{\BB}$ (with associated Baily-Borel strata $\scrT_1,\scrT_2 \subset \bar{\scrS}$), we have that the Zariski-closure of $\scrS_1$ intersects $\scrS_2$ if and only if the inertial rank of $\scrT_1$ is strictly greater than the inertial rank of $\scrT_2$. This is proved similar to Lemma \ref{lem: NOS} following \cite[Lemma 8.4]{PST}, and we omit the argument. 

We now make the following claim. 
\begin{claim}\label{intertialinvariants determines BB stratum}
    Let $\Ann \subset \Dstar$ be any annulus and suppose that $f^{-1}(\VetZl)$ has the property that the rank of inertial invariants is constant. Then $f(\Ann)$ is contained in a the tube over a single Baily-Borel stratum.
\end{claim}
\begin{proof}
    We pick some formal model $\fA \rightarrow \bar{\scrS}$ of $f: \Ann \rightarrow S^{\an}$. The special fiber of $\fA$ is necessarily connected\footnote{We note that $\Ann$ can be replaced by any connected rigid sub-variety of $\Dstar$ - as all we will use for this claim is the special fiber being connected}. It suffices to prove that each irreducible component $C_i$ of the special fiber of $\fA$ maps to a single Baily-Borel stratum. Indeed, mapping $C_i$ to $\scrS^{\BB}$, it suffices to prove that $C_i$ maps to a single open boundary stratum. There is a unique open boundary stratum $\scrS_1$ which contains the image of some open subvariety $U\subset C_i$. If $U\neq C_i$, there is a point $x \in C_i$ which maps to some open boundary stratum $\scrS_2$. It follows that the closure of $\scrS_1$ intersects (and therefore contains) $\scrS_2$, whence the inertial rank of $\scrT_1$ is strictly greater than the inertial rank of $\scrT_2$. This contradicts our hypothesis that the inertial rank of $f^{-1}(\VetZl)$ is constant, and the claim follows. 
\end{proof}

The result \cite[Proposition]{p-adic-borel-Ag} and the argument in the proof of \cite[Theorem 3.4]{p-adic-borel-Ag} immediately proves the statement that the rank of inertial invariants is constant as one varies over a thin annulus. By the claim above, we therefore have that each thin annulus maps to a single Baily-Borel stratum. It suffices to prove that each thick annulus also maps to a single Baily-Borel stratum. Identical to \cite[Section 3.1.2]{p-adic-borel-Ag}, we are reduced to proving the following result. Let $g: {\Gm}_{\F_v} \rightarrow \scrT_{\F_v}$ be a map. Then $g$ extends to a map $\PP^1 \rightarrow \scrT_v$. Composing $g$ by the map $\bar{\scrS} \rightarrow \scrS^{\BB}$, we are reduced to the setting of Theorem \ref{thm: all good or all bad}, and we conclude. 
\end{proof}


\section{Crystalline $p$-adic local systems and analytic prismatic $F$-crystals}\label{sec: prismatic setup}

In this section, we recall the notions of crystalline local systems and $F$-crystals that are used in the article. To simplify notations, we will work in the setting of smooth (formal) schemes. The notion of $F$-crystals and prismatic $F$-crystals extend naturally to the logarithmic setting, which will be used in Section \ref{section:bad_reduction_case}.

\begin{definition} \label{def:isocrystals} Let $X/k$ be a smooth scheme. Let $X_{\textup{crys}}$ denote the $p$-completed crystalline site of $X$,\footnote{In this article, we shall consider the \textit{absolute} crystalline site of $X$, or equivalently, crystalline site of $X$ over the divided power algebra $(W(k), p)$.} equipped with the structure sheaf $\mathcal{O}_{X, \textup{crys}}$.   
\begin{enumerate**}
    \item  By a \textit{crystal} over $X$ we mean a \textit{finite locally free crystal} or equivalently, a \textit{crystal of vector bundles} over $X$, that is, a sheaf of $\mathcal{O}_{X, \textup{crys}}$-modules $\mathbb{E}$ such that for each PD-thickening $(U, T)$ in  $X_{\mathrm{crys}}$, the induced Zariski sheaf $\mathbb{E}_{T}$ is a finite locally free $\mathcal{O}_T$-module, such that for each morphism 
    $g: (U', T') \rightarrow (U, T)$  in $X_{\mathrm{crys}}$, the induced map $    g^* \mathbb{E}_T \isom \mathbb{E}_{T'} $ is an isomorphism. 
    \item  An \textit{isocrystal} over $X$ is an object in the isogeny category of  crystals of modules.  All of the isocrystals we will consider will in fact be obtained from a crystal (in vector bundles) by inverting $p$.    
    \item  An \textit{$F$-crystal}  (resp. \textit{$F$-isocrystal}) over $X$ consists of a pair $(\mathbb{E}, \varphi)$ where $\mathbb{E}$ is a  crystal (resp.  isocrystal) over $X$ and $\varphi$ is an isomorphism 
    \[ \varphi: F_{\mathrm{crys}}^* \mathbb{E} [1/p] \isom  \mathbb{E} [1/p] \] which is compatible with the Frobenius map $F_{\mathrm{crys}}$ on $\mathcal{O}_{X, \textup{crys}}$ induced by functoriality. 
\end{enumerate**}
We write $\Vect^{\varphi}(X_{\textup{crys}})$ (resp. $\textup{Isoc}^{\varphi}(X_{\textup{crys}})$) for the category of $F$-crystals (resp. $F$-isocrystals) over $X$. 
\end{definition}

We also need the notion of prismatic and analytic prismatic $F$-crystals. Let us first recall that, given a $p$-adic formal scheme $\fX/\mathcal{O}_K$, its absolute prismatic site $\fX_{\Prism}$ is the opposite of the category of bounded prisms $(A, I)$ equipped with a map $\Spf A/I \rightarrow \fX$, endowed with the flat topology (on prisms). Let $\mathcal{O}_{\Prism}$ (resp. $\mathcal{I}_{\Prism}$) denote the structure sheaf (resp. the Hodge--Tate sheaf) on $\fX_{\Prism}$, which sends $(A, I) \mapsto A$ (resp. sends $(A, I) \mapsto I$). Let $\varphi_{\Prism}$ denote the Frobenius map on $\mathcal{O}_{\Prism}$. 
 
\begin{example} \label{Example:prisms}
 Let $E = E(u)$ be an Eisenstein polynomial for a fixed uniformizer $\varpi \in \mathcal{O}_K$.  
\begin{enumerate}
    \item The \textit{Breuil--Kisin} prism $(\mathfrak S, E)$ with $\mathfrak S = W (k)[\![u]\!]$ and $\varphi_{\mathfrak{S}} (u) = u^p$ gives an object in $(\Spf \mathcal O_K)_{\Prism}$ via the surjection $\mathfrak S \rightarrow \mO_K$ sending $u \mapsto \varpi$. In fact, by the argument in \cite[Example 2.6(1)]{Bhatt-Scholze-crystal} it covers the final object of the topos $\textup{Shv}((\Spf \mathcal O_K)_{\Prism}, \mathcal O_{\Prism})$. Taking the Cech nerve of $(\mathfrak S, E)$ over the final object in this topos gives rise to a cosimplicial object 
    \begin{equation} \label{eq:simplicial_BK}
     \mathfrak S \mathrel{\substack{\textstyle\longrightarrow\\[-0.6ex] \textstyle\longrightarrow}}  \mathfrak S^{(1)}  \mathrel{\substack{\textstyle\longrightarrow\\[-0.6ex] \textstyle\longrightarrow \\[-0.6ex] \textstyle\longrightarrow}} \mathfrak S^{(2)} 
\end{equation}
     in $(\Spf \mO_K)_{\Prism}$. One can explicitly describe the Prisms $(\mathfrak S^{(i)}, E)$ in terms of the prismatic envelop construction. For example, we have $\mathfrak S^{(1)}= W [\![u, v]\!] \left\{\frac{u - v}{E } \right\}^{\wedge}_{(p, E)}$, where $\{\cdot\}$ denotes the prismatic envelop.  
    \item Let $R =  (\mathcal O_K 
    [t])^{\wedge}_{p}$ be the $p$-adic completion of $\mathcal O_K [t]$ and let $R_0 = (W [t_0])^{\wedge}_p$. Let $\fX = \Spf R$ be the $p$-adic formal $\A^1$ over $\Spf \mathcal O_K$. 
    Let $\mathfrak S_R = R_0 [\![u]\!]$, equipped with a $\delta$-structure given by $\varphi (u) = u^p$ and $\varphi (t_0) = t_0^p$. As in the previous example, we have a surjection $\mathfrak S_R \rightarrow  R$ sending $u \mapsto \varpi$ and $t_0 \mapsto t$, which makes  $(\mathfrak S_R, E)$ into an object in $\mathfrak X_{\Prism}$. 
    \end{enumerate}
\end{example}

\begin{lemma}
    In the second example above, the prism $(\mathfrak S_R, E)$ covers the final object in the topos of $(\fX_{\Prism}, \mathcal O_{\Prism})$.

\end{lemma}

\begin{proof}
 Let $(A, I)$ be a prism in the absolute prismatic site $\fX_{\Prism}$, equipped with a map $\iota: R \rightarrow A/I$. Note that $A$ is canonically a $W$-algebra. Let us pick an element $\widetilde \varpi$ (resp. $\widetilde t$) in $A$ which lifts the image of $\varpi$ (resp. of $t$) in $A/I$ under $\iota$. Let us consider the map $A \rightarrow (A  \otimes_W \mathfrak S_R)^{\wedge}_{(p, I)}$ of $\delta$-rings and form the prismatic envelope (see  \cite[Proposition 3.13]{prism}) 
 \[
  B := (A \otimes_W \mathfrak S_R) \{\frac{u- \widetilde \varpi, t_0 - \widetilde t}{I}\}^{\wedge}_{(p, I)},
 \]
 so we have a map $(A, I) \rightarrow (B, IB)$ of prisms in $\mathfrak X_{\Prism}$. By \cite[Proposition 3.13]{prism}, the map $A \rightarrow B$ is $(p, I)$-completely flat, and in fact $(p, I)$-completely faithfully flat, so $(A, I) \rightarrow (B, IB)$ forms a cover in $\mathfrak X_{\Prism}$.  Finally, note that by construction we have $E(u) = E (\widetilde \varpi) = 0 \mod IB$, so we have $E(u) \in IB$ and thus we have a map of prisms  $(\mathfrak S_R, E(u)) \rightarrow (B, IB)$ by \cite[Lemma 2.24]{prism}. This finishes the proof of the lemma. 
\end{proof}

\begin{definition}[{\cite{Bhatt-Scholze-crystal}}] Let $\fX/ \mathcal{O}_K$ be a smooth $p$-adic formal scheme.    A \textit{prismatic crystal (of vector bundles)} over $\fX$ is an $\mathcal O_{\Prism}$-module $\mathcal E$, such that there exists bounded prisms $(A_i, I_i)$ in $\fX_{\Prism}$, with $\{U_i = \Spf (A_i/I_i)\} $ covering the final object of the topos of $(\fX_{\Prism}, \mathcal O_{\Prism})$, and finite projective $A_i$-modules $E_i$, such that $\mathcal E|_{U_i} \cong E_i \otimes_{A_i} \mathcal O_{\Prism, U_i}$. A \textit{prismatic $F$-crystal} over $\mathfrak X$ is pair $(\mathcal{E}, \varphi_{\mathcal{E}})$, where $\mathcal E$ is a prismatic crystal and $\varphi_{\mathcal{E}}$ is an isomorphism \[ 
\varphi_{\mathcal{E}} : \varphi_{\Prism}^* \mathcal E [1/\mathcal I_{\Prism}] \isom \mathcal E [1/\mathcal I_{\Prism}]
\]    
of $\mO_{\Prism}$-modules. We write $\Vect (\fX_{\Prism}, \mO_{\Prism})$ (resp. $\Vect^{\varphi}(\fX_{\Prism}, \mO_{\Prism})$) for the category of prismatic crystals (resp. prismatic $F$-crystals) over $\fX$. 
\end{definition}

Following notations from \cite{Bhatt-Scholze-crystal}, for a given bounded prism $(A, I)$, we write $\Vect (A)$ for the category of finite projective $A$-modules. We write $\Vect^{\varphi} (A, I)$ for the category of pairs $(E, \varphi_E)$ that consists of a finite projective $A$-module $E$ together with an $A$-linear isomorphism $\varphi_E: \varphi_A^* E [1/I] \isom E [1/I]$, with morphisms being morphisms between the finite projective $A$-modules that are compatible with  Frobenius. By $(p, I)$-completely faithfully flat descent for vector bundles (see \cite[Proposition 2.7]{Bhatt-Scholze-crystal}), we have natural equivalences  
\begin{align*}
     \Vect (\fX_{\Prism}, \mO_{\Prism}) & \isom  \lim_{(A, I) \in \fX_{\Prism}} \Vect (A) \\ 
     \Vect^{\varphi}(\fX_{\Prism}, \mO_{\Prism})  & \isom  \lim_{(A, I) \in \fX_{\Prism}} \Vect^\varphi(A, I)
\end{align*}
In particular, to specify a prismatic $F$-crystal over $\fX$ is equivalent to specifying a prismatic $F$-crystal over  each bounded prism $(A, I) \in \fX_{\Prism}$ in a compatible fashion. If $(A_0, I_0) \in \fX_{\Prism}$ is a bounded prism that covers the final object of $\textup{Shv} (\fX_{\Prism}, \mO_{\Prism})$, this is also equivalent to the data of a prismatic $F$-crystal over $(A_0, I_0)$ which satisfies certain the descent data coming from the Cech nerve of this cover. We will also need the following variant. 

\begin{definition}[{\cite{GuoReinecke}}]  \label{def:Guo_Reinecke_analytic}
\begin{enumerate**}
    \item We define the category of \textit{analytic prismatic $F$-crystals} over   a bounded prism  $(A, I)$, denoted by $\Vect^{\textup{an},\varphi} (A, I)$, to be the category of pairs $(E, \varphi_E)$ where $E$ is a vector bundle over $\spec (A) \backslash V(p, I)$, and $\varphi_E$ is an  isomorphism $\varphi_A^* E[1/I] \isom E[1/I]$. Morphisms in $\Vect^{\textup{an},\varphi} (A, I)$ are morphisms between vector bundles that are compatible with Frobenius.
\footnote{Note that, for a prism $(A, I)$, the Frobenius $\varphi_A$ preserves the zero locus $V(p, I) \subset \spec A$ as well as its complement. Thus the definition above makes sense.} 
    \item  Let $\fX/ \mathcal{O}_K$ be a smooth $p$-adic formal scheme. We define the category $\Vect^{\textup{an}, \varphi} (\fX_{\Prism})$ of \textit{analytic prismatic $F$-crystals} by the derived limit
    \[ 
    \Vect^{\textup{an}, \varphi} (\fX_{\Prism}) := \lim_{(A, I) \in \fX_{\Prism}} \Vect^{\textup{an}, \varphi}(A, I). 
    \]
\end{enumerate**}
\end{definition}
 
    We have a natural forgetful functor 
    $ \Vect^{\varphi} (\fX_{\Prism}) \longrightarrow 
    \Vect^{\textup{an}, \varphi} (\fX_{\Prism}) $  
    from prismatic $F$-crystals to analytic prismatic $F$-crystals. This functor is induced from the map $ \Vect^{\varphi}(A, I) \rightarrow  \Vect^{\textup{an}, \varphi}(A, I)$
    that sends a vector bundle over $\spec A$ to its restriction over the complement of $V(p, I)$ for each $(A, I) \in \fX_{\Prism}$ and is fully faithful. Moreover, it is compatible with crystalline realizations, in the sense that we have a commutative diagram
    \begin{equation} \label{eq:prismatic_F_crystals_restriction_and_realization}
     \begin{tikzcd}
    \Vect^{\varphi} (\fX_{\Prism})     \arrow[r] \arrow[d] &   \Vect^{\textup{an}, \varphi} (\fX_{\Prism}) \arrow[d]\\
    \Vect^{\varphi}(\fX_{s, \textup{crys}}) 
    \arrow[r] & \textup{Isoc}^{\varphi}(\fX_{s, \textup{crys}})
     \end{tikzcd}
    \end{equation}
    where the vertical arrows are induced by specializing to the special fiber $\fX_{s}$ of $\fX$ and identifying the absolute prismatic site $\fX_{s, \Prism}$ of $\fX_s$ with the $p$-completed absolute crystalline site $\fX_{s, \textup{crys}}$ of $\fX_s$ (see \cite[Construction 4.12]{Bhatt-Scholze-crystal}).    Let us consider a special case of this restriction functor in the setting of Example \ref{Example:prisms}.  
\begin{lemma} \label{lemma:an_crystal_to_alg_crystal}  Let $\fX = \Spf R$, where $R = (\mathcal O_K 
    [t])^{\wedge}_{p}$. Let $U = \spec ( \mathfrak S_R)  \backslash V(p, E(u)) $ be the open subset of $\spec (\mathfrak S_R)$ and denote by $ j: U \rightarrow \spec (\mathfrak S_R) $ the open immersion. The essential image of the fully faithful functor
    \[ 
    \Vect^{\varphi} (\fX_{\Prism}) \longrightarrow 
    \Vect^{\textup{an}, \varphi} (\fX_{\Prism}) 
    \]
consists of pairs $(\mathcal E, \varphi_{\mathcal E})$ satisfying the following condition: if we write $(E_U, \varphi_{U})$ for the vector bundle over $U$ obtained by evaluating $(\mathcal E, \varphi_{\mathcal E})$ on the prism $(\mathfrak S_R, E(u)) \in \fX_{\Prism}$, then $j_* E_U$ is a vector bundle over $\spec \mathfrak S_R$. 
\end{lemma}

\begin{proof}
    This is \cite[Proposition 1.26]{IKY}. 
\end{proof}

\begin{definition} \label{def:crys_local_sys} Let $\fX/\mO_K$ be a smooth $p$-adic formal scheme. Let $\fX_{\eta}/K$ be its adic generic fiber and let $\fX_s/k$ denote the special fiber. Let $\L$ be an \'etale $\Z_p$-local system on $\fX_{\eta}$. We say that $\L$ is \textit{crystalline} if there exists an $F$-isocrystal $(\mathbb{E}, \varphi)$ over $\fX_s$, together with a Frobenius equivariant isomorphism of $\mathbb{B}_{\mathrm{crys}}$-vector bundles
    \[ 
     \mathbb{B}_{\mathrm{crys}} (\mathbb{E}) \isom \mathbb{B}_{\mathrm{crys}} \otimes_{\Z_p} \L.
    \] 
    Here $\mathbb{B}_{\mathrm{crys}}(\mathbb{E})$ is the sheaf of $\mathbb{B}_{\mathrm{crys}}$-modules on the pro-\'etale site $\fX_{\eta,\textup{pro\'et}}$ associated to $(\mathbb{E}, \varphi)$. We denote the category of crystalline $\Z_p$-local systems on $\fX_{\eta}$ by $\textup{Loc}^{\textup{crys}}_{\Z_p}(\fX_{\eta})$. 
    \footnote{Note that, \textit{a priori}, the notion of crystallinity of local systems depends on the integral model $\fX$ over $\mathcal{O}_K$. In fact, this notion only only depends on the generic fiber $\fX_{\eta}$ (and independent of the chosen integral model).}
\end{definition}

We shall need the following result in the article, due to Bhatt--Scholze \cite{Bhatt-Scholze-crystal} in the case of $\Spf \mO_K$ and to Du--Liu--Moon--Shimizu \cite{DLMS1}/Guo--Reinecke \cite{GuoReinecke} in general. 

\begin{theorem}[{\cite{Bhatt-Scholze-crystal, DLMS1, GuoReinecke}}]\label{thm:prismatic_F_crystals_and_local_systems}
Let $\fX/\mO_K$ be a smooth $p$-adic formal scheme. There is a natural equivalence of categories 
\[ 
    \Vect^{\textup{an}, \varphi} (\fX_{\Prism}) \isom \textup{Loc}^{\textup{crys}}_{\Z_p}(\fX_{\eta}). 
\] 
In particular, this equivalence is functorial in $\fX$. 
\end{theorem}

\begin{remark}
    In the setting above, if $\L$ is a crystalline $\Z_p$-local system on $\fX_{\eta}$, then one can recover the $F$-isocrystal over the special fiber $\fX_{s}$ from this equivalence via the crystalline realization (see Diagram (\ref{eq:prismatic_F_crystals_restriction_and_realization})).  
\end{remark}

\begin{example} \label{Example:prisms_revisit} Let us revisit Example \ref{Example:prisms} once again.  
     \begin{enumerate**}
         \item Let $(\mathfrak S, E)$ be the Breuil--Kisin prism, and let $U_0 = \spec (\mathfrak S) \backslash V(p, E)$ denote the open subscheme obtained as the complement of a closed point in  $\spec \mathfrak S$.   
         Theorem \ref{thm:prismatic_F_crystals_and_local_systems} (due to Bhatt--Scholze in this context) says that a $\Z_p$-lattice $\Lambda$ in a crystalline $\textup{Gal}_K$-representation is equivalent to the data of a vector bundle $\mE_0$ over $U_0 = \spec (\mathfrak S) \backslash V(p, E)$, equipped with a Frobenius isomorphism after inverting $E$, as well as the descent data coming from the cosimplicial complex (\ref{eq:simplicial_BK}). Since $\mathfrak S$ is a regular local scheme of dimension 2, $\mE_0$ uniquely extends to a vector bundle over $\spec \mathfrak S$, which still carries the Frobenius isomorphism after inverting $E$, and again satisfies descent. In other words, for $\Spf \mO_K$, the fully faithful embedding $
    \Vect^{ \varphi} ((\Spf \mO_K)_{\Prism}) \rightarrow  
    \Vect^{\textup{an}, \varphi} ((\Spf \mO_K)_{\Prism}) $ is an equivalence. Moreover, via Theorem \ref{thm:prismatic_F_crystals_and_local_systems} and  this equivalence, a $\Z_p$-lattice $\Lambda$ in a crystalline representation is equivalent to a \textit{Breuil--Kisin module} $(\mathfrak M, \varphi_{\mathfrak M})$ that satisfies descent along the cosimplicial complex (\ref{eq:simplicial_BK}), that is, equipped with a descent isomorphism after pulling back to $\mathfrak S^{(1)}$ which satisfies a cocycle condition over $\mathfrak S^{(2)}$.
   \item Let $R =  (\mO_K [t]^{\wedge}_{p})$ be as in Example \ref{Example:prisms} (2) and let $\fX = \Spf R$, so $\fX_{\eta} = \sD = \textup{Spa} (R[1/p], R)$ is the closed unit disc over $K$. In this case, Theorem \ref{thm:prismatic_F_crystals_and_local_systems} tells us that crystalline $\Z_p$-local system on $\sD$ is equivalent to the category of analytic prismatic $F$-crystals over $\fX$. In particular, it gives rise to a vector bundle $E_U$ over $U = \spec (\mathfrak S_R) \backslash V(p, E(u))$ equipped with a Frobenius $\varphi_E$ (by evaluating the analytic prismatic $F$-crystal on the prism $(\mathfrak S_R, E(u))$). 
   \end{enumerate**}
\end{example}

\section{From $F$-isocrystals to $F$-crystals}\label{Sec:pointwiseconstancy}

Let $\mathbb L$ be a crystalline $\Z_p$-local system on the closed unit disc $\sD$, and let $\mathbb E$ be the $F$-isocrystal over the special fiber of $\sD$ (which is a copy of $\A^1_{k}$) attached to $\mathbb L$. Let us also recall that, if $V$ is a crystalline $\Q_p$-representation of the Galois group $\textup{Gal}_K$ with nonnegative Hodge--Tate weights, and $\Lambda \subset V$ is a $\textup{Gal}_K$-stable $\Z_p$-lattice, then one can attach to $\Lambda$ a Frobenius module 
\[\mathbb D_{\textup{crys}} (\Lambda) = (M, \varphi_M),\] which consists of a finite free $W(k)$-module $M$ together with a Frobenius map $\varphi_M: M \rightarrow M$ that becomes an isomorphism upon inverting $p$. This can be achieved by considering the Breuil--Kisin module $(\mathfrak M, \varphi_{\mathfrak M})$ and base changing along the map $\mathfrak S = W(k) [\![u]\!] \rightarrow W(k)$ sending $u \mapsto 0$. We may regard $\mathbb D_{\textup{crys}} (\Lambda)$ as 
an $F$-crystal over $\spec k$ (see Example \ref{Example:prisms_revisit} and also see Diagram (\ref{eq:prismatic_F_crystals_restriction_and_realization})).  The goal of this section is to show the following. 

\begin{theorem} \label{prop:existence_of_crystal}
   Assume the above setup.  Up to replacing $\sD$ by a smaller 
   closed disc in $\sD$ if necessary, there exists an $F$-crystal $\mathbb D$ over $\A^1$ with $\mathbb D [1/p] \cong \mathbb E$, that is compatible with the crystalline $\Z_p$-local system  $\mathbb L$ in the following sense: for every finite extension $L/K$ and every classical $L$-point $x \in \sD$ which specializes to a closed point $\overline{x} \in \A^1$,  there is an isomorphism $\mathbb D|_{\overline{x}} \cong \mathbb D_{\textup{crys}} (\mathbb L|_{x})$  of $F$-crystals over $\overline{x}$. Consequently, up to replacing $\sD$ by a smaller disc, the isomorphism class of the $F$-crystal $\mathbb{D}_{\textup{cris}}(\bL_x)$ is independent of the classical point $x\in \sD$.
\end{theorem}   


\subsection{Locally free extensions}
  Let $\mathcal T$ be an admissible $p$-adic formal scheme over $\Spf W$, and consider the sheaf of rings $\cO_{\mathcal T}[\![u]\!]$.\footnote{In the context of \Cref{Example:prisms}(2), the reader may take $\mathcal T = \Spf R_0$ as a working example in this subsection.} We say $\mathcal T$ is regular if every local ring is regular.  In this section we prove the following:

\begin{proposition}\label{prop:blow-up to resolve}
    Let  $\mathcal T$ be a regular 2-dimensional $p$-adic formal scheme over $\Spf W$ and $E$ a finitely generated $\cO_{\mathcal T}[\![u]\!]$-module.  Then there is a sequence \[
    \mathcal T'=\mathcal T_n\to\cdots \to\mathcal T_1=\mathcal T\] of formal admissible blow-ups at closed points, such that $(f[\![u]\!]^*E)^{\vee\vee}$ is locally free as an $\mathcal O_{\mathcal T} [\![u]\!]$-module, where $f: \mathcal T' \rightarrow \mathcal T$ denotes the composition of sequence of maps above.  
\end{proposition}

We begin with the following observation.

\begin{lemma}\label{lem:vanishing}Let $\cO$ be a 3-dimensional regular local ring and $E$ a finitely generated reflexive $\cO$-module.  Then $\Ext^i_\cO(E,\cO)=0$ for $i>1$.
\end{lemma}
\begin{proof}
 Every term in the double dual spectral sequence \[R\Hom_\cO(R\Hom_\cO(E,\cO),\cO)=E\]
 vanishes except $E^{\vee\vee}$, $\Ext^1_\cO(E^\vee,\cO)$, and $\Ext^3_\cO(\Ext^1_\cO(E,\cO),\cO)$.  In particular, this implies that $\Ext^3_\cO(\Ext^i_\cO(E,\cO),\cO)=0$ for $i>1$, so $\Ext^i_\cO(E,\cO)=0$ by local duality.
\end{proof}
The main step of \Cref{prop:blow-up to resolve} is the following:
\begin{lemma}\label{lem:strictly decreasing}Let $\cO$ be a 3-dimensional regular local ring and $E$ a finitely generated reflexive $\cO$-module which is not free.  Let $f:X\to\Spec\cO$ be the blow-up along any regular curve.  Then we have 
\[\ell(\sExt^1_{\cO_X}((f^*E)^\vee,\cO_X))<\ell(\Ext^1_\cO(E^\vee,\cO)),\]
where $\ell $ denotes the length of ($\mathcal O$-)modules. 
\end{lemma}
\begin{proof}
\noindent\emph{Step 1.}  $E$ has a presentation of the form
\begin{equation}\label{eq:present}\begin{tikzcd}0\ar[r]&\cO^m\ar[r,"A"]&\cO^n\ar[r]& E\ar[r]& 0\end{tikzcd}.\notag\end{equation}
\begin{proof}
    Consider any presentation
\begin{equation}\label{eq:present2}\begin{tikzcd}\cO^m\ar[r,"A"]&\cO^n\ar[r]& E\ar[r]& 0\end{tikzcd}.\notag\end{equation}
We claim that $N=\img(A)$ is free.  Indeed, applying $RHom_\cO(-,\cO)$ to the sequence
\[\begin{tikzcd}0\ar[r]&N\ar[r]&\cO^n\ar[r]& E\ar[r]& 0\end{tikzcd}\]
and using \Cref{lem:vanishing}, we see that $\Ext^i_\cO(N,\cO)=0$ for $i>0$.
\end{proof}

\noindent\emph{Step 2.}  Let $M$ be any nonzero finitely-generated $\cO$-module supported on (some thickening of) the closed point $x$ of $\Spec\cO$.  Then $L^if^*M=0$ for $i<-1$ and $H^0(\sExt_{\cO_X}^2(L^{-1}f^*M,\cO_X))\neq 0$.
\begin{proof}First observe that the vanishing claim is true for $M=k(x)$ using the Koszul resolution.  In general we may take a short exact sequence 
\[\begin{tikzcd}
    0\ar[r]&M'\ar[r] &M\ar[r]&k(x)\ar[r]&0.
\end{tikzcd}\]
and the vanishing follows from the above observation by induction on $\ell(M)$.

For the second claim, again observe that the claim is true for $M=k(x)$ since $L^{-1}f^*k(x)=\cO_C(-1)$ where $C=f^{-1}(x)\cong \P^1_{k(x)}$ and $\sExt^2_{\cO_X}(\cO_C(-1),\cO_X)=\cO_C$.  In general, using the same sequence we have an exact sequence
\[
\begin{tikzcd}
0\ar[r]&L^{-1}f^*M'\ar[r]&L^{-1}f^*M \ar[r]& \cO_C(-1)
\end{tikzcd}
\]
and the image of the rightmost map is therefore either 0 or $\cO_C(-a)$ for some $a\geq 1$.  Thus by induction we may assume we are in the latter case.  But then we have an inclusion
\[\begin{tikzcd}
    0\ar[r]&\cO_C(a-1)\cong\sExt^2_{\cO_X}(\cO_C(-a),\cO_X)\ar[r]& \sExt^2_{\cO_X}(L^{-1}f^*M,\cO_X)
\end{tikzcd}\]
whence the claim.
\end{proof}
Taking the dual of the presentation from Step 1 we have
\[\begin{tikzcd}
    0\ar[r]&E^\vee\ar[r]&\cO^n\ar[r,"A^*"]&\cO^m\ar[r]&\Ext^1_\cO(E,\cO)\ar[r]&0.
\end{tikzcd}\]
Let $F=\img(A^*)$ and $Q=\Ext^1_\cO(E,\cO)$.
\vskip1em
\noindent\emph{Step 3.}  We have $L^if^*F=0$ and $L^if^*(E^\vee)=0$ for $i<0$.
\begin{proof}Pulling back the sequence
\[\begin{tikzcd}
    0\ar[r]&F\ar[r]&\cO^m\ar[r]& Q\ar[r]&0
\end{tikzcd}\]
and using the vanishing in Step 2 implies the first claim.  Pulling back 
\[\begin{tikzcd}
    0\ar[r]&E^\vee\ar[r]&\cO^n\ar[r]& F\ar[r]&0
\end{tikzcd}\]
and using the first claim implies the second.
\end{proof}
\noindent\emph{Step 4.}    There is a natural short exact sequence
\[\begin{tikzcd}
    0\ar[r]&f^*(E^\vee)\ar[r] &(f^*E)^\vee\ar[r]&L^{-1}f^*Q\ar[r]&0.
\end{tikzcd}\]
\begin{proof}Since $\cO_X$ is torsion-free, pulling back the presentation from Step 1 we have an exact sequence
\[\begin{tikzcd}
0\ar[r]&\cO_X^m\ar[r,"f^*A"]&\cO_X^n\ar[r]&f^*E\ar[r]&0
\end{tikzcd}\]
and therefore an exact seqeunce
\[\begin{tikzcd}
0\ar[r]&(f^*E)^\vee\ar[r]&\cO_X^n\ar[r,"f^*A^*"]&\cO_X^m.
\end{tikzcd}\]
Now applying the previous step to the diagram with exact diagonals
\[\begin{tikzcd}[/tikz/cells={/tikz/nodes={shape=asymmetrical
  rectangle,text width=1.5cm,text height=2ex,text depth=0.3ex,align=center}},row sep=tiny, column sep=tiny]
{}&0\ar[dr]&{}&{}&{}&{}&{}\\
{}&{}&L^{-1}f^*Q\ar[dr]&{}&0&{}&{}\\
    {}&{}&{}&f^*F\ar[ur]\ar[dr]&{}&{}&{}\\
    {}&{}&\cO_X^n\ar[ur]\ar[rr,"f^*A^*",swap]&{}&\cO_X^m\ar[rd]&{}&{}\\
   {}& f^*(E^\vee)\ar[ur]&{}&{}&{}&f^*Q\ar[rd]&{}\\
    L^{-1}f^*F\ar[ur]&{}&{}&{}&{}&{}&0
\end{tikzcd}\]
yields the claim.
\end{proof}

\noindent\emph{Step 5.}  There is a natural short exact sequence
\[\begin{tikzcd}
    0\ar[r]&H^1(f^*(E^\vee)^\vee)\ar[r] &\Ext^1_{\cO}(E^\vee,\cO)\ar[r]&H^0(\sExt_{\cO_X}^1(f^*(E^\vee),\cO_X))\ar[r]&0.
\end{tikzcd}\]
In particular, $\ell(\sExt_{\cO_X}^1(f^*(E^\vee),\cO_X))\leq \ell(\Ext^1_\cO(E^\vee,\cO))$.
\begin{proof}Use
\[Rf_*(R\sHom_{\cO_X}(Lf^*(E^\vee),\cO_X))=R\Hom_\cO(E^\vee,Rf_*\cO_X)\]
together with $Rf_*\cO_X=\cO$ and the vanishing of $L^{i}f^*(E^\vee)$ for $i<0$ from Step 3.
\end{proof}

\noindent\emph{Step 6.}  \[\ell(\sExt^1_{\cO_X}((f^*E)^\vee,\cO_X))\leq \ell(\Ext^1_\cO(E^\vee,\cO))-\ell(H^0(\sExt^2_{\cO_X}(L^{-1}f^*Q,\cO_X))).\]
\begin{proof}
    Applying $R\Hom_{\cO_X}(-,\cO_X)$ to the sequence from Step 4 and using the vanishing from \Cref{lem:vanishing} we have an exact sequence
\[\begin{tikzcd}
    0\ar[r]&\sExt^1_{\cO_X}((f^*E)^\vee,\cO_X)\ar[r]& \sExt^1_{\cO_X}(f^*(E^\vee),\cO_X)\ar[r]& \sExt^2_{\cO_X}(L^{-1}f^*Q,\cO_X)\ar[r]&0.
\end{tikzcd}\]
The first term has dimension 0, so the sequence remains exact on taking global sections
\[\begin{tikzcd}[sep=small]
    0\ar[r]&H^0(\sExt^1_{\cO_X}((f^*E)^\vee,\cO_X))\ar[r]& H^0(\sExt^1_{\cO_X}(f^*(E^\vee),\cO_X))\ar[r]& H^0(\sExt^2_{\cO_X}(L^{-1}f^*Q,\cO_X))\ar[r]&0.
\end{tikzcd}\]
Combining this with Step 5, we have the claim.
\end{proof}
\noindent\emph{Step 7.}  Conclusion of proof.
\vskip1em
By the assumption on $E$ and \Cref{lem:vanishing} we have $Q\neq 0$.  But then by the nonvanishing in Step 2 we have $0\neq \ell(H^0(\sExt^2_{\cO_X}(L^{-1}f^*Q,\cO_X)))$, so the claim follows from Step 6.
\end{proof}

\begin{proof}[Proof of \Cref{prop:blow-up to resolve}]
 Note it is equivalent to show $(f[\![u]\!]^*E)^\vee$ is locally free.  Since $E^\vee$ is a reflexive $\mathcal O_{\mathcal T} [\![u]\!]$-module, we have \[ \sExt^i_{\mathcal O_{\mathcal T} [\![u]\!]}(E^\vee,\mathcal O_{\mathcal T} [\![u]\!])=0\] for $i>1$ by \Cref{lem:vanishing}. Moreover,   $\sExt^1_{\mathcal O_{\mathcal T} [\![u]\!]}(E^\vee,\mathcal O_{\mathcal T} [\![u]\!])$ is supported at finitely many closed points, hence has finite length $\ell$.  We proceed by induction on $\ell$, the $\ell=0$ case being trivial.  Let $x \in \mathcal T$ be any point in the support, $\cO_{\mathcal T,x}$ the local ring of $\mathcal T$ at $x$ with maximal idea $\frak m_x$, and let $\frak n_x$ be the ideal generated by $\frak m_x$ in $\cO_{\mathcal T,x}[\![u]\!]$.  Let 
 \[ \pi:B\to \Spec\cO_{\mathcal T,x} \] be the blow-up at $\frak m_x$, and note that: 
 \begin{itemize} 
     \item  
 the base-change $\pi[\![u]\!]:B[\![u]\!]\to \Spec (\cO_{\mathcal T,x} [\![u]\!])$ is the blow-up at $\frak n_x$, where $B[\![u]\!]$ denotes the scheme theoretic base change $B[\![u]\!] := B \times_{\Spec (\cO_{\mathcal T,x}} \Spec (\cO_{\mathcal T,x} [\![u]\!])$;
 \item  the $p$-adic completions of the local rings of $B$ are identified with the local rings of the formal admissible blow-up 
 \[ g:\mathcal T'\to \mathcal T \] of $\mathcal T$ at $x$.  
 \end{itemize}   In particular, for a closed point $x'$ of $\mathcal T'$ (which we also view as a closed point in $B [\![u]\!]$), the $p$-adic complete local ring 
 \[ (\mathcal O_{\mathcal T'} [\![u]\!])_{x'} = \mathcal O_{\mathcal T', x'} [\![u]\!] \]   is flat over $\cO_{B[\![u]\!],x'}$. Here  $(\mathcal O_{\mathcal T'} [\![u]\!])_{x'}$ denotes the stalk of the sheaf $\mathcal O_{\mathcal T'} [\![u]\!]$ of the $p$-adic formal scheme $\mathcal T'$ at $x'$ 

According to \Cref{lem:strictly decreasing}, the length of $\sExt^1_{\mathcal O_{B[\![u]\!], x'}}((\pi[\![u]\!]^*E_x)^\vee,\cO_{B [\![u]\!], x'})$ is strictly smaller than \[ \ell_x := \sExt^1_{\mathcal O_{\mathcal T, x} [\![u]\!]}(E_x^\vee,\mathcal O_{\mathcal T, x} [\![u]\!]),\] so the same is true of $\sExt^1_{ (\mathcal O_{\mathcal T'} [\![u]\!])_{x'}}((g[\![u]\!]^*E_x)^\vee, (\mathcal O_{\mathcal T'} [\![u]\!])_{x'})$.  Therefore, we know that 
\[ \ell \big(  \sExt^1_{ \mathcal O_{\mathcal T'} [\![u]\!]}((g[\![u]\!]^*E)^\vee,\mathcal O_{\mathcal T'} [\![u]\!]) \big) < \ell.\] 
By induction the proof is complete.
\end{proof}

\subsection{Existence of the Prismatic $F$-crystal} 

We are now ready to prove Theorem \ref{prop:existence_of_crystal}. 

\begin{proof}[Proof of Theorem \ref{prop:existence_of_crystal}]
Let $\mathfrak X = \Spf R = \Spf (\mathcal O_K [t])^{\wedge}_{p}$ be an integral model of $\sD$ as in the setting of Example \ref{Example:prisms_revisit} (2). In particular, the crystalline local system $\mathbb L$ gives rise to a vector bundle $E_U$ over $U = \spec (\mathfrak S_R) \backslash V(p, E(u))$, where we recall that $\mathfrak S_R = R_0 [\![u]\!]$ and $R_0 = (W [t_0])^\wedge_p$. Let $\fX_0 = \Spf R_0$. 
Let $E = j_* E_U$ where $j: U \rightarrow \spec \mathfrak S_R$ is the open immersion, which is a reflexive module over $\mathcal O_{\mathfrak X_0} [\![u]\!]$. By Proposition \ref{prop:blow-up to resolve}, we know that there exists successive admissible blowups of $\mathfrak X_0$ at closed points such that the pullback  $\pi[\![u]\!]^*E$ is a locally free $\mathcal O_{\mathfrak X_0'} [\![u]\!]$-module, where $\pi: \mathfrak X_0' \rightarrow \mathfrak X_0$ is the composition of the required  blowups. From the proof of Proposition \ref{prop:blow-up to resolve}, we know that by shrinking the radius of the disc in the generic fiber at each step of the blowup if necessary, we may without loss of generality assume that at each step we only need to blowup at the origin. In other words, the required blowups $\pi$ can be achieved by the map $\pi_n: R_0 = (W [t_0])^\wedge_p \rightarrow R_n = (W [t_n])^\wedge_p$ sending $t_0 \mapsto p^n t_n$ for some large enough integer $n$. This means that, after replacing the disc $\sD$ by its closed subdisc $\sD_n = \{|z| \le 1/p^n\}$, the local system $\mathbb L|_{\sD_n}$ gives rise to a pair $(E_{n, U}, \varphi_{n, U})$ where $E_{n, U}$ is a vector bundle over $U_n = \Spec R_n [\![u]\!] \backslash V(p, E(u))$ that extends to a vector bundle over $\Spec R_n [\![u]\!]$. By Lemma \ref{lemma:an_crystal_to_alg_crystal}, $\L|_{\sD_n}$  gives rise to a (unique up to isomorphism) prismatic $F$-crystal $(\mathcal E_n, \varphi_n)$ over $\Spf (\mathcal O_K [t'])^\wedge_p$ where $t' = p^n t$. This prismatic $F$-crystal is compatible with the $F$-crystal $\mathbb D_{\textup{crys}} (\mathbb L|_{x})$ under specialization to closed points $x \in \sD_n$ by the functoriality of the equivalence in Theorem \ref{thm:prismatic_F_crystals_and_local_systems}. This finishes the proof of the theorem. 
\end{proof}

\section{Constancy of the $F$-crystal} \label{sec: constancy}
In this section, we will prove that the fiber-wise $F$-crystal produced in Section \ref{Sec:pointwiseconstancy} is actually constant. Our argument is essentially Lemmas 1.10 and 1.9 of \cite{oortleaves} in the setting of $F$-crystals.

\begin{lemma}[Oort, Lemma 1.10]\label{Lemmaoort1.10}
    Let $\D_1$ and $\D_2$ be two $F$-crystals over $\Fpbar$ having the same rank and let $N$ be a fixed integer. Then there are finitely many sub-crystals of $\D_2$ isomorphic to $\D_1$ whose co-kernel is $p^N$ torsion.
\end{lemma}
\begin{proof}
    Let $L = \Hom(\D_1,\D_2)$. Consider the set $S \subset L$ defined as $S = \{f\in L: \Im(f) \supset p^N \D_2 \}$. Let $f_1,f_2 \in S$ be two maps such that $f_1 \equiv f_2 \mod p^N$ in $L/p^N$. A direct computation shows that $\Im f_1 = \Im f_2$. It follows that the number of sub-crystals is bounded by the image of the subset $S$ in $L/p^N$. The lemma follows from the fact that $L$ is a finite generated $\Z_p$ module.
\end{proof}

\begin{definition}
    Let $X/\Fpbar$ be a scheme and $\D$ an $F$-(iso)crystal.  Let $\pt_X:X\to \Spec\Fpbar$ be the map to a point, and for any point $x\in X(\Fpbar)$ let $i_x:\Spec\Fpbar\to X$ be the corresponding morphism.
    \begin{enumerate}
    \item We say $\D$ is \emph{constant} if there is an (iso)crystal $\D_0$ on $\Spec\Fpbar$ such that if $\D\cong \pt_X^*\D_0$.
    \item We say $\D$ is \emph{point-wise constant} if there is an (iso)crystal $\D_0$ on $\Spec\Fpbar$ such that $i_x^*\D\cong \D_0$ for every $x\in X(\Fpbar)$.
    \end{enumerate}
\end{definition}
We are now ready to prove 
\begin{theorem}\label{thm: constancy of F-crystal}
Let $X$ be a smooth connected variety over $\Fpbar$. Let $\D /X$ be a point-wise constant $F$-crystal such that the $F$-isocrystal $\D[1/p]$ is constant. Then $\D$ itself is constant.
\end{theorem}

\begin{proof}
    Let $\D_0  $ be the $F$-crystal $i_x^*\D$ for $x\in X(\Fpbar)$. The constancy of $\D[1/p]$ gives an isogeny $g: \D \rightarrow \pt_X^* \D_0$. By Lemma \ref{Lemmaoort1.10}, the set $X(\Fpbar)$ admits a partition into finitely many subsets $X_1 \hdots X_m$ such that for two points $x,x' \in X_i$, we have $g_x(\D_x) = g_{x'}(\D_{x'})$. Without loss of generality, suppose that $X_1$ is Zariski dense, and let $x\in X_1$ be any point. Consider the inclusion of crystals $g': \pt_X^*\D_0 \rightarrow \pt_X^*\D_0$, where $g' = \pt_X^* g_x$. We now have that $\Im g'$ and $\Im g$ are two subcrystals of $\pt^*\D_0$ which agree on a Zariski-dense set of points. We will now show that they are the same. 
    
    We first reduce to the case that $X$ is affine.  We claim that for two crystals $\D_1,\D_2$ on $X$ and any dense Zariski open $U\subset X$, any morphism $f_U:\D_1|_U\to \D_2|_U$ extends uniquely to a morphism $f:\D_1\to \D_2$.  By the uniqueness, such an extension patches, so we may assume $X$ is affine, and in particular lifts to a formally smooth scheme $\fX/ W(\Fpbar)$ equipped with a lift of Frobenius $\phi:\fX\to\fX$.  Then $\D_1$ corresponds to a flat vector bundle $E_1$ on $\fX$ together with a flat morphism $\phi_{E_1}:\phi^*E_1\to E_1$ for which $\phi_{E_1}[1/p]$ is an isomorphism; likewise for $\D_2$ and $E_2$.  Then $f$ yields a morphism $\ff_\fU:E_1|_{\fU}\to E_2|_{\fU}$.  In characteristic 0, it is easy to see that $\ff_\fU[1/p]$ extends uniquely to a moprhism $E_1|_{\fX[1/p]}\to E_2|_{\fX[1/p]}$.  But then the underlying morphism of vector bundles $E_1\to E_2$ is defined outside of a set of codimension 2, hence extends uniquely, and the extension is compatible with the connection and the Frobenius structure. 

    Therefore, we suppose that $X = \Spec A$. Let $\tilde{A}$ be a $p$-adically complete smooth $W$-algebra with a lift of Frobenius such that $A = \tilde{A} \mod p$. Evaluating all these crystals on $\tilde{A}$, we obtain two Frobenius-stable flat-subbundles of $\D_0\otimes \tilde{A}$ which: 
    \begin{enumerate}
        \item Are the same after inverting $p$.
        \item Agree at a dense set of $W$-points.
    \end{enumerate}
    Even without using $F$-crystal structure, the formal-smoothness of $\tilde{A}$ implies that the these two sub-bundles must already agree. The theorem follows.

\end{proof}

\section{Proof of the extension theorem for $\Dstar$ in the good reduction case}\label{sec: Dstar}

In this section, we prove our main results for $\sD^\times$, assuming $\Im(f)$ intersects the good reduction locus of $X_F$.  This is a necessary assumption in the period image case; in the Shmiura variety case, we include the good reduction case argument separately as a simpler case of the general strategy.

\subsection{Extending crystalline local systems}\label{sec: crystalline extension}
The following result is Theorem 1.16 of \cite{DY}. 
\begin{theorem}[\cite{DY}]\label{thm: crystalline extension}
    Let $\bL/\Dstar$ be a $\Z_p$-local system with crystalline fibers everywhere. Then, $\bL$ extends to a crystalline local system on $\sD$. 
\end{theorem}
For the sake of completeness, we will sketch a proof in the easy case where in addition there exists a smooth formal scheme $\fX$, a crystalline $\mathbb Z_p$-local system $\bL'/\fX_\eta$, and a map $g: \Dstar \rightarrow \fX_\eta$ such that  $\bL = g^* \bL'$. Let $\mathbb{E}[1/p]$ denote the isocrystal associated to $\bL'$ on $\frak{X}_p$. Consider the flat bundle underlying the filtered flat bundle $\mathcal{RH}_p(\bL)/\Dstar$ (resp. $\mathcal{RH}_p(\bL')/\fX_\eta$)---we will somewhat abusively use the same notation both for the filtered flat bundles and the underlying flat bundles. We have that $\mathbb{E}[1/p](\fX)$ is canonically isomorphic to $\mathcal{RH}_p(\bL')$ (see for eg \cite[Lemma 3.7]{bst-integral-canonical}). It suffices to prove that $\mathcal{RH}_p(\bL)$ extends to a flat bundle on $\sD$, by \cite[Theorem 5.7]{p-adic-borel-Ag}. We will in fact use the fact that $\bL$ is pulled back from $\fX_\eta$ to conclude that $\mathcal{RH}_p(\bL)|_{\sD'}$ is actually the trivial flat bundle for every closed sub-disc\footnote{In fact, we do not need to restrict to $\sD'$ for this triviality if $\fX$ satisfies the property that every map $\mathbb{A}^1 \rightarrow \fX_p$ extends to a map $\mathbb{P}^1\rightarrow \fX_p$. We note that $\integralS$ does satisfy this property.} of $\sD' \subsetneq \sD$. Indeed, it suffices to prove that $\mathcal{RH}_p(\bL)|_{\Ann'_n}$ is the trivial flat bundle for every integer $n$, where $\Ann'_n := \Ann_n \cap \sD'$. We pick an integral model $\frak{g'}_n: \mathfrak{A}'_n \rightarrow  \mathfrak{A}_{n+1} \rightarrow \fX$ of the composite map $g_n':=\Ann'_n \subsetneq \Ann_n \xrightarrow{g|_{\Ann_{n+1}}}\fX_\eta$. By the appendix of \cite{p-adic-borel-Ag}, we may assume that the irreducible components of the reduced special fiber $C_{n+1}$ (resp. $C'_n$) of $\frak{A}_{n+1}$ (resp. $\frak{A}'_n$) consist of (two) $\mathbb{A}^1$s and (an unspecified but finite number of) $\mathbb{P}^1$s, with the further condition that the there is no combinatorial monodromy. By \cite[Remark A.7]{p-adic-borel-Ag}, there is a marked point $x_i$ on each of the two $\mathbb{A}^1$s in $C_{n+1}$ such that no point in the open annulus $\Ann^\circ_{n+1}$ specializes to a point in $\mathbb{A}^1 \setminus \{ x_i\}$. It follows that the image of $C'_n$ in $C_{n+1}$ is contained in a proper (possibly reducible) curve, and in particular, $\frak{g}'_n$ extends to the compactification $\bar{C'}_n$ of $C'_n$, whose irreducible components are now solely $\mathbb{P}^1$s (and there is still no combinatorial monodromy). Every $F$-isocrystal on $\bar`{C}'_n$ is trivial, and therefore $\frak{g}^{\prime,*}_n(\mathbb{E}[1/p])$ is trivial. Evaluating on the thickening given by $\frak{A}'_n$, we have that $\frak{g}^{\prime,*}_n(\mathbb{E}[1/p])(\frak{A}'_n)$ is trivial. The result now follows from the chain of isomorphisms $\frak{g}^{\prime,*}_n(\mathbb{E}[1/p])(\frak{A}'_n) \simeq g^{\prime,*}_n(\mathbb{E}[1/p](\fX)) \simeq g^{\prime,*}_n(\mathcal{RH}_p(\bL)) \simeq \mathcal{RH}_p(\bL|_{\Ann'_n})$.

A very similar argument also works in the setting of geometric period images. Indeed, let $j$ be the smallest integer such that the image of $\Ann_n$ is contained in $(Y^j)^{\an}$. Then, there is a finite subset $\Xi \subset \Ann_n$ whose complement maps to $(Y^{j})^{\an}\setminus (Y^{j-1})^{\an}$. As $S^j \rightarrow Y^j$ is an isomorphism away from $Y^{j-1}$, $\Ann_n \setminus \Xi$ lifts to $(S^j)^{\an}$. The fact that $S^j \rightarrow Y^j$ is proper allows us to extend this to a map $\Ann_n \rightarrow (S^{j})^{\an}$. Now, the identical argument outlined above works. 

\subsection{The Shimura case}

We have the following proposition. 
\begin{proposition}\label{prop: constantFLconstantmap}
    Let $g:C\rightarrow \integralS_p$ be a map where $C$ is a curve over $\Fpbar$. Suppose that $g^{*}\Vcris$ is constant on $C$. Then the map $g$ is constant. 
\end{proposition}
\begin{proof}
We first replace $C$ by its image -- note that the $F$-crystal remains constant. By replacing $C$ by an open subset, we may assume that $C$ is a smooth curve. We pick a smooth lift $\tilde{C}/W$. Consider $g^* \Vdr$, the filtered flat bundle on $S$ pulled back to $C$. The underlying flat bundle is obtained by taking the mod $p$ reduction of the flat bundle $(g^* \Vcris)(\tilde{C})$. The filtration is simply the kernel of (powers of) Frobenius mod $p$. Therefore, we have that $g^*\Vdr$ is constant as a filtered flat bundle. This implies that the Kodaira-Spencer map is trivial on $C$, which contradicts the versality of the Kodaira-Spencer map. The proposition follows. 
\end{proof}

We are now ready to prove the one-dimensional disk case of \Cref{thm:introextension} for a Shimura variety in the good reduction case.

\begin{theorem}\label{disk Shimura}
    \Cref{thm:introextension} is true for $a=1,b=0$ and $X$ a Shimura variety, assuming $\Im(f)$ intersects the good reduction locus of $X_F$.
\end{theorem}
\begin{proof}
We will prove that the image of $\Dstar$ lands in a residue disc after suitably shrinking $\sD$ -- this would yield the theorem. The proof consists of assembling the various results already proved.

 We have that $f^* \bL$ extends to a local system with crystalline fibers on $\sD$. By Proposition \ref{prop:existence_of_crystal}, we may shrink $\sD$ and obtain an $F$-crystal $\D$ on $\A^1 \bmod p$ such that for any classical $y\in \sD$, $\D_{\cris}(\bL_y) = \D_{\bar{y}}$. By further shrinking $\sD$, we get that the $F$-crystal $\D$ is constant, and therefore the isomorphism class of $\D_{\cris}(\bL_y)$ is independent of the point $y$. 

Let $\Ann_n \subset \sD$ be a closed annulus with outer radius 1 and inner radius $\frac1{p^n}$, and let $f_n: \mathfrak{A}_n \rightarrow \integralS$ be an integral model for the map $f|_{\Ann_n}$. By \cite[Appendix]{p-adic-borel-Ag}, we have that the special fiber $C$ of $\mathfrak{A}_n$ is a union of two affine lines and some number of projective lines with trivial combinatorial monodromy. It suffices to show that $f_n|_{C_i}$ is constant for every component $C_i$ of $C$. 

Let $\bar{y} \in C_i$ and let $y\in \mathfrak{A}^{\rig}_n(K) = \Ann_n$ be a point that specializes to $\bar{y}$. We have that $_{\cris}\V_{\bar{y}} \simeq \D_{\cris(\bL_y)}$ and therefore we have that $f_n^*(_{\cris}\V)_{C_i}$ is point-wise constant. Let $\bar{C}_i$ be the compactification of $C_i$. The map $C_i \rightarrow \integralS_p$ extends to $\bar{C}_i \rightarrow \integralS_p$ the results of Section \ref{Sec: bdry}. Therefore, the $F$-crystal $f_n^*( \Vcris)|_{C_n}$ extends to an $F$-crystal on $\bar{C}_i$, and therefore the $F$-isocrystal must be constant -- indeed every $F$-isocrystal on $\P^1$ is constant. We are now in the situation of a point-wise constant $F$-crystal such that the underlying $F$-isocrystal is constant. By \autoref{thm: constancy of F-crystal}, we have that $f_n^*( \Vcris)|_{C_n}$ is constant. The theorem follows by applying \autoref{prop: constantFLconstantmap}.
\end{proof}

\subsection{Period images}

We have the following result. 
\begin{proposition}\label{prop: pointwise crystal}
    Fix some integer $j$ and consider $\pi^j: S^j \rightarrow Y^j$. Let $y\in \integralY^j(\Fpbar)$ denote some point and let $\cF_y \subset \cS^j$ denote the fiber over $y$. Then we have that $\Ucris^j|_{\cF_y}$ is point-wise constant. Further, the isomorphism class of this crystal is $\D_{\cris}({\VetZp}_{\tilde{y}})$, where $\tilde{y} \in Y(\cO_K)$ is any lift of $y$. 
\end{proposition}
\begin{remark}
    Note that this proposition assigns a canonical crystal to every $y\in \integralY(\Fpbar)$. 
\end{remark}
\begin{proof}
    Let $T(\cF_y)$ denote the tube in $\cS^j$ over $\cF_y$. For any lift $\tilde{y}$ of $y$, let let $F_{\tilde{y}} \subset T(\cF_y)$ denote the fiber of $\pi^j$ over $\tilde{y}$. As the stratification is uniform, there exists a lift $\tilde{y}$ of $y$, such the specialization map $F_{\tilde{y}}\rightarrow \cF_y$ is surjective. Let $\tilde{z}\in T(\cF_y)$ be any point and let $z\in \cF_y(\Fpbar)$ denote its specialization. As $\UetZp^j$ is crystalline, we have a canonical isomorphism of crystals $\D_{\cris}({\UetZp^j}_z) \rightarrow {\Ucris^j}_{z}$. 

    Let $z_i,z_2 \in \cF_y$. Let $\tilde{z}_1$ and $\tilde{z}_2$ be points in $F_{\tilde{y}}$ which specialize to $z_1$ and $z_2$ respectively. By the above remark, it suffices to prove that $\D_{\cris}({\UetZp^j}_{\tilde{z}_1})$ is isomorphic to $\D_{\cris}({\UetZp^j}_{\tilde{z}_2})$. But this follows directly from the fact that $\pi^j(\tilde{z}_1) = \pi^j(\tilde{z}_2)$ and that $\UetZp$ is simply the pull-back of $\VetZp$ under $\pi^j$. The proof of the second part now follows from this and the isomorphism $\D_{\cris}({\UetZp^j}_{\tilde{z}}) \rightarrow {\Ucris^j}_z$. 
\end{proof}

We are now ready to prove the one-dimensional extension theorem in this setting.
\begin{theorem}  \label{disk period}
    \Cref{thm:introextension} is true for $a=1,b=0$ and $X$ a geometric period image.
\end{theorem}
\begin{proof}
    As in the Shimura case, we will prove that the image of $\Dstar$ lands in a residue disc after suitably shrinking $\sD$. Also as in the Shimura case, we may further shrink $\Dstar$ such that the isomorphism class of $\D_{\cris}(\bL_y)$ is independent of the point $y\in \Dstar$. 
    We let $\Ann_n$, $\fA_n$, $f_n$, and $C$ be as in the Shimura case. Pick some component $C_0$ of $C$ isomorphic to $\P^1$. Suppose that $j$ is the smallest integer such that $f_n(C_0) \subset \integralY^j$. Then, we have that $C_0$ generically maps into $\integralY^j \setminus \integralY^{j-1}$ and therefore we may lift $f_n$ to a map $g_n: C_0\rightarrow \cS^j$. By \autoref{prop: pointwise crystal}, we have that $g_n^{*}(\Ucris^j)$ is point-wise constant. As $C_0$ is isomorphic to to $\P^1$, we also have that the iso-crystal $g_n^{*}(\Ucris^j[1/p])$ is constant. By \autoref{thm: constancy of F-crystal}, we have that $g_n^{*}(\Ucris^j)$ is constant. As in the Shimura case, it follows that $g_n^{*}(\UFL^j)$ is trivial as a filtered flat bundle. Up to passing to a finite cover $C'_0$ of $C_0$, we may lift the map $g_n$ to a map $h_n: C'_0 \rightarrow \cT^j$. We therefore have that $h_n^{*}(q^{i*}(\Udr^j))$ is trivial as a filtered flat bundle, as $\Udr^j$ and $\UFL^j$ pull back to isomorphic filtered flat bundles on $\cT^j$. Therefore, the Griffiths bundle is trivial on $C'_0$. If $f_n|_{C_0}$ is non-constant, we also have that Griffiths bundle on $C_0$ associated to $g_n^{*}(\Udr^j) \simeq f_n^*(\Vdr)$ is ample, and therefore the Griffiths bundle must be ample on $C'_0$. This is a contradiction. Therefore, $f_n|_{C_0}$ must be constant. It follows that the map $f_n$ contracts every component isomorphic to $\P^1$. By \cite[Appendix Remark A.7]{borel}, it follows that the open subannulus of $\Ann_n$ maps to a single residue disc. This holds for every $n$, and therefore we have that the open punctured disc maps to a single residue disc. The theorem follows. 
\end{proof}


\section{Proof of the general extension theorem for $\Dstar$ in the Shimura case} \label{section:bad_reduction_case}

In this section we prove the following:
\begin{theorem}\label{thm:borel boundary}
    \Cref{thm:introextension} is true for $a=1,b=0$ and $X$ a Shimura variety.
\end{theorem}







\subsection{Log crystals on the Shimura variety and its boundary}

\subsubsection{Weight filtrations on isocrystals}
Let $(Y,M_Y)$ be a log scheme over $\overline{\F}_p$.  

\begin{definition}
 Let  $(Y,M_Y)_{\CRIS}$ be the  $p$-completed big logarithmic crystalline site where we equip $\Z_p$ with  the trivial log structure (cf. \cite[Appendix B]{DLMS2}). By an (iso)crystal we shall exclusive mean an (iso)crystal in vector bundles on this site, as in \Cref{sec: prismatic setup}.   $F$-crystals and $F$-isocrystals are defined similarly as in Definition \ref{def:isocrystals}. 
 \end{definition}

Let $R=W\langle x_1,\ldots,x_a,y_1,\ldots ,y_b\rangle$, $\fX^\mathrm{std}=\Spf R$ and $\fD^\mathrm{std}=V(y_1\cdots y_b)\subset \fX^\mathrm{std}$.  Let $\fX/\Spf W$ be a smooth $p$-adic formal scheme and let $\fD \subset \fX$ be a relative strict normal crossings divisor---that is, a pair $(\fX,\fD)$ which is \'etale-locally isomorphic over $W$ to a disjoint union of $(\fX^\mathrm{std},\fD^\mathrm{std})$ for possibly different values of $a,b$ and for which each irreducible component of $\fD$ is regular.  By strata of $(\fX,\fD)$ we mean irreducible components of intersections of irreducible components of $\fD$.    Let $\fD'$ be a locally closed union of strata, which for simplicity we assume to have pure codimension 1.  We let $\fX^\LOG:=(\fX,M_{\fX})$ denote the log $p$-adic formal scheme associated to $(\fX,\fD)$ over the base $W$ where $W$ is given the trivial log structure, so $M_\fX$ is the monoid of functions which are invertible away from  $\fD$. We let $\fD^\LOG:=(\fD,M_\fD)$ (resp. ${\fD'}^\LOG:=(\fD',M_{\fD'})$) denote $\fD$ (resp. $\fD'$) with the log structure pulled back from $\fX^\LOG$. We let $\fX^\LOG_p$ (resp. $\fD_p,\fD'_p,\fX_p^\LOG, \fD_p^\LOG,\fD_p'^\LOG$) be the mod $p$ special fibers of $\fX$ (resp. $\fD,\fD',\fX^\LOG, \fD^\LOG,\fD'^\LOG$).  Note that $\fX^\LOG$ is smooth over $\Spf W$. 

Let $\breve \fD'^\LOG$ denote the $p$-completed PD envelope of $\fD'^\LOG_p \subset\fX^\LOG$. By \cite[Theorem 6.2]{kato}, evaluating on $\fX^\LOG,\breve \fD'^\LOG$ gives the following equivalences of categories. 
\begin{enumerate}\label{properties:crystals}
    \item Crystals on $\fX_p^\LOG$ (i.e. on $(\fX^\LOG_p)_{\CRIS}$) and flat bundles with quasinilpotent connection on $\fX^\LOG$. Note that such an object is just a flat vector bundle on $\fX$ with logarithmic singularities along $\fD$ and nilpotent $p$-curvature mod $p$. 

    \item Crystals on $\fD_p'^\LOG$ and logarithmic flat bundles with quasinilpotent connection on $\breve \fD'^\LOG$.     
\end{enumerate}

Now suppose we have a flat vector bundle $(V,\nabla)$ on $\fD'^\LOG$ (which may or may not lift to $\breve \fD'^\LOG$). Consider the composition
\[V\xrightarrow{\nabla} V\otimes\hat\Omega_{\fD'^\LOG/W}\to V\otimes\left(\hat\Omega_{\fD'^\LOG/W}/\hat \Omega_{\fD'/W}\right).\]
Since $$\left(\hat\Omega_{ \fD'^\LOG/W}/\hat \Omega_{ \fD'/W}\right)\cong \nu_*\cO_{\tilde \fD'}$$
where $\nu:\tilde \fD'\to \fD'$ is the normalization of $\fD'$, for each stratum $\Sigma$ of $\fD'$ we therefore obtain $\mathrm{codim}\,\Sigma$ commuting endomorphisms of $V|_\Sigma$, indexed by the divisors containing $\Sigma$, which we call the internal residues.  Note that the category of crystals on $\fD'_p$ is equivalent to the full subcategory of crystals on $\fD_p'^\LOG$ with vanishing internal residues.

Given an $F$-crystal $\D$ on $\fD_p'^\LOG$ with Frobenius $\phi_\D:F^*\D\to \D$ and a flat subbundle $U\subset \D(\fD'^\LOG)$, we say that $U$ is $\phi_\D$-stable if the following is true.  There is a natural flat sub-vector bundle $F^*U\subset (F^*\D)\fD'^\LOG)$ which, after passing to a Zariski cover for which there is a global lift of Frobenius $\tilde F:\fD'^\LOG\to \fD'^\LOG$, agrees with the pullback $\tilde F^*U\subset \tilde F^*\D(\fD'^\LOG)=(F^*\D)(\fD'^\LOG)$.  We then say that $U$ is $\phi_\D$-stable if $\phi_\D(\fD'^\LOG)(F^*U)\subset U$.  This condition is easily checked to be independent of the choices.

In the following, we say a flat vector bundle $(V,\nabla)$ on $\fX^\LOG/W$ is algebraic if there is a projective log smooth pair $(\bar\fX,\bar\fD)/W$ containing $(\fX,\fD)$ as a Zariski open and to which $(V,\nabla)$ extends as $(\bar V,\nabla)$.  If $(V,\nabla)$ has nilpotent residues, we assume the same for the extension. A similar definition is made for the adic generic fiber $\fX^\LOG[1/p]$ viewed as a log adic space. A flat sub-vector bundle $U\subset V|_{\fD'}$ will be algebraic if $\fD'$ extends to a locally closed union of strata $\bar\fD'\subset\bar\fD$ and $U$ extends to a flat sub-vector bundle of $\bar V|_{\bar \fD'}$.  We have the following result.

\begin{proposition}\label{prop: weight filt on pullback crystal}
We work with the above setup.  Let $\D_{\fX_p^\LOG}$ be an $F$-crystal on $ \fX_p^\LOG$ with nilpotent residues and for which the flat vector bundle $\D_{\fX_p^\LOG}(\fX^\LOG)$ is algebraic.  Let $\D :=\D_{\fX_p^\LOG}|_{ \fD_p'^\LOG}$  be the restriction and $\phi_\D:F^*\D\to\D$ its Frobenius. Let $U\subset \D(\fD'^\LOG)$ be an algebraic flat sub-vector bundle which is locally a summand on the level of coherent sheaves. Then there exists a sufficiently large integer $m$ such that $(F^m)^*U$ is the evaluation of a (unique) sub crystal $\D_U\subset(F^m)^*\D|_{\fD_p'^\LOG}$ which is also locally a summand on the level of coherent sheaves.  If moreover $U$ is $\phi_\D$-stable, then $\D_U$ is a sub $F$-crystal.
\end{proposition}

\begin{proof}Since $\phi_\D$ is flat and the Frobenius morphism on $(F^m)^*\D$ is $(F^m)^*\phi_\D$, the second statement is clear.  For the first statement, first note the following:  

\begin{lemma}\label{lem:complex analytic}
   Assume that $(\fX, \fD)/W$ is a projective log smooth pair for this lemma. 
   Let $(V,\nabla)$ be a flat vector bundle on $\fX^\LOG[1/p]$ with nilpotent residues and $U\subset (V,\nabla)|_{\fD'^\LOG)[1/p]}$ a   flat sub-bundle.  Then $U$ lifts to a unique  flat sub-vector bundle $\hat U\subset \hat V$ over the $\fD'[1/p]$-adic of $\fX^\LOG[1/p]$.
\end{lemma}

\begin{proof} 
By taking an embedding $K\subset\C$, we may prove the corresponding statement over $\C$.  Clearly the lift is unique if it exists.  The lift exists in the completion along any stratum supported on $\fD'$, as in this case there is Zariski-locally a retraction.  Moreover, the lift exists analytically in a tubular neighborhood of $\fD'$.  It follows that the lift exists in the completion along $\fD'$. 
\end{proof}

Back to the proof of the proposition. The extension is clearly unique if it exists, so we may freely pass to a Zariski cover.  Thus, we may assume an \'etale morphism $f:(\fX,\fD)\to(\fX^\mathrm{std},\fD^\mathrm{std})$.  Let $\tilde F:(\fX^\mathrm{std},\fD^\mathrm{std})\to (\fX^\mathrm{std}, \fD^\mathrm{std})$ be the standard lift of Frobenius $x_i\mapsto x_i^p$, $y_i\mapsto y_i^p$.  We also denote by $\tilde F:(\fX, \fD)\to (\fX, \fD)$ the unique lift through $f$.  

Given the above equivalences, we are reduced to the following local statement, where we don't assume $\fX$ is projective:
\begin{lemma}\label{lem: upper triangular}
    Let $(\cO^{m+n},\nabla)$ be an algebraic flat vector bundle on $\fX^\LOG$ with nilpotent residues such that $\nabla$ is upper triangular (with respect to the splitting $\cO^{m+n}=\cO^m\oplus\cO^n$) in restriction to $\fD'^\LOG$.  Then for fixed sufficiently large $m\in\mathbb{N}$ there is a unique matrix $P=\left(\begin{smallmatrix}
        1&0\\
        M&1
    \end{smallmatrix}\right)$ defined over the $p$-adic completion $\breve \fD'$ of the PD envelope of $\fD'$ in $\fX$, with $M=0$ when restricted to $\fD'$, such that 
    $P^{-1}((\tilde F^m)^*\nabla) P$ is upper triangular.
\end{lemma}
\begin{proof}We follow the proof of \cite[II.2.13]{sabbah}.  The assumption on the connection means that writing $\nabla=d+A$ and 
\[A=\sum_i A_idx_i+\sum_j B_j\frac{dy_j}{y_j}\]
that the matrices $A_i$ and $B_j$ are of the form $\left(\begin{smallmatrix}
        *&*\\
        0&*
    \end{smallmatrix}\right)$ mod $(y_1\cdots y_b)$.  
    
    Since the connection is algebraic, by \Cref{lem:complex analytic}, such a $P$ exists after inverting $p$ over the $\fD'$-adic completion, and it is easily seen to be unique.

    Write $A_i=\sum A_{ik}(y_1\cdots y_b)^k$ and likewise for $B_j$ and $M$.  Note $M_0=0$.   We have
    \[P^{-1}\nabla P = P^{-1}dP+P^{-1}AP\]
    and so the $\frac{dy_j}{y_j}$ term of the upper triangular condition is 
    \[y_j\frac{\partial}{\partial y_j}M-MB_j^{11}-MB_j^{12}M+B_j^{21}+B_j^{22}M=0\]
where we write $C=\left(\begin{smallmatrix}
C^{11}&C^{12}\\C^{21}&C^{22}    
\end{smallmatrix}\right)$ for any matrix $C$.  Thus,
\[kM_k-M_kB_{j0}^{11}+B_{j0}^{22}M_k=-B_{jk}^{21}+\sum_{\substack{\alpha+\beta=k\\\alpha<k}} M_\alpha B_{j\beta}^{11}-\sum_{\substack{\alpha+\beta=k\\\beta<k}} B_{j\alpha}^{22}M_\beta+\sum_{\alpha+\beta+\gamma=k} M_\alpha B_{j\beta}^{12}M_\gamma.\]
which is defined mod $(y_1\cdots y_b)$.  Note that the last term has no summands with $\alpha$ or $\gamma$ equal to $k$.  Let $M_{k|j}$ be the restriction to $y_j=0$, and likewise for the other terms.  Since $B_{j0|j}^{11}$ and $B_{j0|j}^{22}$ are nilpotent, the morphism $\nu_{j}:N\mapsto -NB_{j0|j}^{11}+B_{j0|j}^{22}N$ is nilpotent of nilpotency index $\leq (\mathrm{rk}\; V)^2$.  Then $(1+\nu_j/k)^{-1}$ is a polynomial in $\nu_j/k$ of order $<(\mathrm{rk}\; V)^2$ and we have
\begin{equation}\notag\label{formula for term}M_{k|j}=\frac{1}{k}(1+\nu_j/k)^{-1}\left(-B_{jk|j}^{21}+\sum_{\substack{\alpha+\beta=k\\\alpha<k}} M_{\alpha|j} B_{j\beta|j}^{11}-\sum_{\substack{\alpha+\beta=k\\\beta<k}} B_{j\alpha|j}^{22}M_{\beta|j}+\sum_{\alpha+\beta+\gamma=k} M_{\alpha|j} B_{j\beta|j}^{12}M_{\gamma|j}\right)\end{equation}

Thus, setting $r=\mathrm{rk}\; V$ and assuming $v(B_{k|j})\geq b\geq 0$, we have 
\[v(M_{k|j})\geq r^2(b-v(k))+\min\left(\min_{\alpha<k}v(M_{\alpha|j}),\min_{\alpha+\beta\leq k}\left(v(M_{\alpha|j})+v(M_{\beta|j})\right)\right).\]
Applying the above analysis to the connection $(F^m)^*\nabla$, we can take $b=m$, and further have that $B_{jk}=0$ unless $v(k)\geq m$.   The above formula for $M_{k|j}$ implies it is also true that $M_{k|j}=0$ unless $v(k)\geq m$.  Set $a(k):=\frac{1}{r^2}v(M_{p^mk|j})$, so that
\[a(k)\geq -v(k)+\min\left(\min_{i<k}a(i),\min_{i+j\leq k}(a(i)+a(j))\right).\]
\begin{claim}
    $a(k)\geq -\frac{2}{p-1}k$.
\end{claim}
\begin{proof}
    We prove by induction that $a(k)\geq v(k)-2v(k!)$.  This is clearly true for $k=1$, since $a(1)=0$.  For $i<k$ we have 
    \[-v(k) +\left(v(i)-2v(i!)\right)\geq v(k)-2v(k!).\]
    For $i+j<k$ we have
    \[-v(k)+ \left(v(i)-2v(i!)\right)+\left(v(j)-2v(j!)\right)\geq v(k)-2v(k!). \]
    For $i+j=k$, note that
    \[v\left({\textstyle\binom{k}{i}}\right)\geq v(k)-\frac{v(i)+v(j)}{2}.\]
    Indeed, this is clear from $v(k)=\min(v(i),v(j))$ if $v(i)\neq v(j)$, and otherwise we may scale and assume $v(i)=v(j)=0$.  Since $v\left(\binom{k}{i}\right)$ is the number of carries adding $i$ and $j$ in base $p$, this is at least the number of terminal 0s in $k$ in base $p$.  Thus,
    \begin{align*}-v(k)+ (v(i)-2v(i!))+(v(j)-2v(j!))&=v(k)-2v(k!)+\left(-2v(k)+v(i)+v(j)+2v\left({\textstyle\binom{k}{i}}\right)\right)\\
    &\geq v(k)-2v(k!).\end{align*}
    To finish, we have $a(k)\geq -2v(k!)$ and $v(k!)\leq \frac{k}{p-1}$, so the claim is proven.
\end{proof}
Now, we have $v(M_{p^mk|j})\geq -\frac{2r^2}{p-1}k$ while on the other hand
\[v((p^mk)!)=\frac{p^m-1}{p-1}k+v(k!).\]
This is true for each $j$, so for $m$ sufficiently large, $v(M_{k})+v(k!)$ tends to infinity and the lemma is proven.
\end{proof}
\end{proof}

Let $(V,\nabla)$ be a flat bundle on $\fD'^\LOG[1/p]$ with nilpotent residues.  On each stratum $\Sigma$ (equipped with its induced log structure), each internal residue endomorphism $N:V|_\Sigma\to V|_\Sigma$ has an associated weight filtration $W(N)_\bullet V|_\Sigma$ by flat subbundles.  It is uniquely determined by the properties that $N(W(N)_kV|_\Sigma)\subset W(N)_{k-2}V|_\Sigma$ for all $k$ and induces $N^k:\gr^{W(N)}_{k}V|_\Sigma\xrightarrow{\cong}\gr_{-k}^{W(N)}V|_\Sigma$ for all $k$.  We say $V$ is local monodromy compatible if there is a global filtration $W_\bullet V$ of $V$ restricting to the weight filtration of each residue on each stratum.  In this case we call $W_\bullet V$ the weight filtration.

We say an isocrystal $\D[1/p]$ on $\fD_p'^\LOG$ with nilpotent residues is local monodromy compatible if the evaluation $V$ on $\fD'^\LOG[1/p]$ is, and we say a crystal $\D$ is local monodromy compatible if the associated isocrystal $\D[1/p]$ is.  If $W_\bullet V$ comes from a filtration $W_\bullet \D[1/p]$ of the isocrystal $\D[1/p]$ (resp. from a filtration $W_\bullet\D$ of the crystal $\D$ which is locally split on the level of coherent sheaves), we call $W_\bullet \D[1/p]$ (resp. $W_\bullet \D$) the weight filtration of $\D[1/p]$ (resp. $\D$).  It is evidently unique.  In particular, $W_\bullet\D$ is $\D\cap W_\bullet\D[1/p]$ if it exists, and $W_\bullet \D[1/p] $ (resp. $W_\bullet \D$) is a filtration by $F$-isocrystals (resp. $F$-crystals) if $\D[1/p]$ (resp. $\D$) is an $F$-isocrystal (resp. $F$-crystal).
\begin{corollary}
    Let $\D_{\fX_p^\LOG}$ be an $F$-crystal on $\fX_p^\LOG$ with nilpotent residues, and assume its restriction $\D:=\D_{\fX_p^\LOG}|_{\fD_p'^\LOG}$ is local monodromy compatible.  Then $\D[1/p]$ admits a weight filtration $W_\bullet \D[1/p]$.  Moreover, the associated graded $\gr^W_\bullet \D[1/p]$ is the pullback of an $F$-isocrystal on $\fD'_p$ (with trivial log structure).
\end{corollary}
\begin{proof}
    The first claim follows from \Cref{prop: weight filt on pullback crystal}.  The second claim follows since an isocrystal on $\fD_p'^\LOG$ with trivial residues is the pullback of an isocrystal on $\fD'_p$.
\end{proof}

\subsubsection{Weight filtrations on crystals for Shimura varieties}
In the following, over the reflex field $E$ we take a log smooth toroidal compactification $\bar S$ of $S$, which admits  morphism $\bar S\to S^{\BB}$. $S^{\BB}$ is naturally stratified by Shimura varieties, and we refer to their preimages $T$ in $\bar S$ as Baily--Borel strata.  Up to a (toroidal) modification of $\bar S$, we may assume each such stratum has pure codimension 1.  We let $\bar S^\LOG$ be $\bar S$ equipped with its natural log structure and $T^\LOG$ the induced log structure on $T$.  Finally, let $\V_\Hdg$ denote the variation of Hodge structure whose underlying local system corresponds to the adjoint representation of $G$.  We denote by $\bar\V_\Hdg$ the Deligne extension equipped with its logarithmic connection and the Schmid extension of the Hodge filtration.  It is naturally a filtered flat vector bundle on $\bar S^\LOG$.

\begin{lemma}\label{good intersections}
\begin{enumerate}
    \item Let $\bar\V_\Hdg|_{T^\LOG}$ be the restriction to a filtered log flat vector bundle on $T^\LOG$.  Then $\bar\V_\Hdg|_{T^\LOG}$ is local monodromy compatible in the sense of the previous section.  In particular there is a global weight filtration $W_\bullet\bar\V_\Hdg|_{T^\LOG}$.
    \item Let $\Sigma\in T$ be a stratum and let $\{N_i\}_{1\leq i\leq k}$ be the residues associated to the divisors meeting $\Sigma$.  Then for any $a\in \Q_{\geq 0}^k$ the weight filtration of $\sum a_iN_i$ is $W_\bullet\bar\V_\Hdg|_{T^\LOG}|_{\Sigma}$.
    \item The Hodge filtration $\Fil^\bullet\bar\V_\Hdg|_{T^\LOG}$ intersects the weight filtration $W_\bullet \bar\V_\Hdg|_{T^\LOG}$ typically..
\end{enumerate}
    
\end{lemma}
\begin{proof}
    The first claim is standard, and follows since the internal local monodromy of a stratum preserves the weight filtration and is trivial in the graded.  The second is also standard, and follows because the nearby cycles mixed Hodge structure along $\Sigma$ is graded polarized by $q(-,N_i^\bullet-)$ for each $i$.  For the third, note that $\Fil^\bullet \bar\V_\Hdg$ has (at most) nonzero degrees $-1,0,1$, and $W_\bullet\bar\V_\Hdg$ has at most nonzero degrees $-2,-1,0,1,2$.  Since $(W_\bullet\bar\V_\Hdg,\Fil^\bullet\bar\V_\Hdg)$ pointwise admits an integral structure (given by the iterated nearby cycles functor) that is graded polarized by $q(-,N^\bullet-)$ and which has Hodge weights $(p,q)$ satisfying $-1\leq p,q\leq 1$, $\Fil^{1}$ surjects onto $\gr^W_2$ and intersects $W_1$ typically.  Further, the image of $\Fil^1\cap W_1$ projects to a Lagrangian subspace (that is, a maximal isotropic) of $\gr_1^W$, so $\Fil^1\cap W_0$ is typical.  The remaining intersections follow from duality.   
\end{proof}

\begin{remark}\label{retraction 1}
    \Cref{good intersections} is the first instance of the existence of a retraction onto the boundary in the case of a Shimura variety.  It means that, for a sufficiently small tubular neighborhood $U$ of a Baily--Borel stratum $T$, the weight and Hodge filtrations give a variation of mixed Hodge structures on $U\backslash T$.  In particular, the associated graded variation on $T$ (which comes via pullback from the map to the corresponding stratum $S_T^{\BB}\subset S^{\BB}$ of the Baily--Borel compactification) extends to $U$, and therefore defines an extension of the map $T\to S_T^{\BB}$ to $U\to S_T^{\BB}$.  In the setting of a general variation of Hodge structures, we do not obtain a variation of mixed Hodge structures on $U\backslash T$ because $\Fil^\bullet$ and $W_\bullet$ intersect atypically along $T$.
\end{remark}

We may spread out $\bar \scrS$ and $\scrS^{\BB}$ and assume the former is a log smooth compactification $\bar \scrS$ with a morphism $\bar \scrS\to \scrS^{\BB}$ over $\cO_E[1/n]$.  We define Baily--Borel strata in the obvious way, and equip them with their natural log structures.  We let $\fS$ denote the $p$-adic completion of $\scrS$, and likewise for $\bar\fS$.  We may further assume \Cref{good intersections} gives a filtration $W_\bullet \D(\fT^\LOG)$ of the evaluation $\D(\fT^\LOG)$ by flat subbundles which is locally split on the level of coherent sheaves and which intersects $\Fil^\bullet \D(\fT^\LOG)$ typically.  

We have the following crucial result. 
\begin{proposition}\label{prop: weight filtration crystal}
Let $\fT^\LOG\subset \fS^\LOG$ be a Baily-Borel stratum and set $\D=\bar\V_\cris|_{\fT_p^\LOG}$.  Then $\D$ admits a weight filtration $W_\bullet \D$.  Moreover, the associated graded $\gr^W_\bullet \D$ is the pullback of an $F$-crystal on $\fT_p$ (with trivial log structure). 
\end{proposition}
\begin{proof}
    For simplicity, denote $\D=\bar\V_\cris|_{\fT_p^\LOG}$ with Frobenius $\phi_\D:F^*\D\to\D$.  For any $m$ denote $^mW_\bullet:=(F^m)^*W_\bullet \D(\fT^\LOG)$, which is a filtration of $((F^m)^*\D)(\fT^\LOG)$ by flat subbundles which is locally split on the level of coherent sheaves.  Note that since $^0W_\bullet$ is $\phi_\D$-stable, for any $m$ the saturation of $((F^{m-1})^*\phi_\D)(\fT^\LOG)\left(^mW_\bullet\right)$ is $^{m-1}W_\bullet$. By \Cref{prop: weight filt on pullback crystal}, there is some integer $m_0$, such that for any $m\geq m_0$, $^{m}W_\bullet$ is the evaluation of a filtration $W_\bullet (F^m)^*\D$ of the $m$-fold Frobenius pullback by sub-$F$-crystals which is locally split on the level of coherent sheaves.   
    \begin{claim}For all $m$, $^mW_\bullet$ is the evaluation of a filtration $W_\bullet (F^m)^*\D$ by sub-$F$-crystals which is locally split on the level of coherent sheaves.
    \end{claim}
    \begin{proof}We prove the claim by descending induction on $m$, the base case having been discussed above.  If the claim is true for any particular $m$, by the above the saturation of $((F^{m-1})^*\phi_\D)(W_\bullet (F^m)^*\D)$ gives the desired filtration $W_\bullet(F^{m-1})^*\D $ provided we show it is locally split on the level of coherent sheaves.  
    
    In the following we think of our crystals in terms of their evaluations on $\breve\fT^\LOG$.  To check the above condition we may pass to a Zariski cover and thereby assume a global lift of Frobenius $\tilde F:\breve\fT^\LOG\to \breve\fT^\LOG$.  Since $\D$ is the restriction of a log Fontaine--Laffaille module, its evaluation on $\breve\fT^\LOG$ comes equipped with a filtration $\Fil^\bullet\D(\breve\fT^\LOG)$ with respect to which $\phi_\D(\breve\fT^\LOG)$ is strongly divisible, meaning it extends to an isomorphism
        \[\phi_\D(\breve\fT^\LOG):\sum_i p^{-i}\tilde F^*\Fil^i\D(\breve\fT^\LOG)\to \D(\breve\fT^\LOG).\]
        It follows that $((F^{m})^*\phi_D)(\breve\fT^\LOG)$ is strongly divisible with respect to $(\tilde F^m)^*\Fil^\bullet\D(\breve\fT^\LOG)$ for any $m$.  But $(\tilde F^m)^*\Fil^\bullet\D(\breve\fT^\LOG)|_{\fT^\LOG}$ and $^mW_\bullet$ intersect typically for any $m$, since the same is true for $m=0$ and both are pulled backs under $\tilde F|_{\fT^\LOG}^m$.  As intersecting typically is an open property, if the claim is true for any particular $m$, then $(\tilde F^m)^*\Fil^\bullet\D(\breve\fT^\LOG)$ and $(W_\bullet(F^m)^*\D)(\breve\fT^\LOG)$ intersect typically, so the saturation of the image of $(W_\bullet(F^m)^*\D)(\breve\fT^\LOG)$ under $((F^{m-1})^*\phi_\D)(\breve\fT^\LOG)$ is locally split on the level of coherent sheaves.
    \end{proof}

\end{proof}


\subsection{Weakly admissible subobjects}

We have the following result. 
\begin{lemma}\label{lem: existence of weight filtration Qp local system fiber}Let $\fT\subset\bar\fS$ be a Baily--Borel stratum.
    Let $\tilde{x}\in \fS^{\rig}(K)$ be a classical point that specializes to $x\in \fT(\Fpbar)$. Under the natural identification $\bar{\V}_{\cris}[1/p]|_x(K_0) \simeq D_{\mathrm{st}}({\VetZp}|_{\tilde{x}}[1/p])$, the filtration $({W_\bullet}\bar\V_\cris|_{\fT^\LOG}[1/p])|_x (K_0)$ of $ \bar\V_\cris[1/p]|_x(K_0) $ from \Cref{prop: weight filtration crystal} equals the filtration induced by $N$. Further, this filtration is in the category of weakly admissible submodules, and therefore yields a filtration of $W_{\bullet}\VetZp|_{\tilde x}$ of $ \VetZp|_{\tilde{x}}$ such that the associated graded Galois representation is crystalline. 
\end{lemma}

To prove this result we will need the following lemmas. The setting is over a point. 

\begin{lemma}\label{lem: selfdual filtered F isocrystal}
    Let $\D$ be a self dual filtered $F$-isocrystal for which the Newton polygon is above the Hodge polygon.  Then the Newton and Hodge polygons have the same endpoint.
\end{lemma}
\begin{proof}
    The Newton/Hodge polygons of the dual are obtained by taking the reflection of the Newton/Hodge polygon through the $y$-axis and translating the left endpoint to the origin.  Thus, the Newton polygon is above the Hodge polygon for the dual iff the Newton polygon is above the Hodge polygon when they are shifted to have the same right endpoint.  Thus, if Newton is above Hodge for both the isocrystal and its dual, the right endpoints must coincide.
\end{proof}
\begin{lemma}\label{N GT} Let $\D$ be the filtered log $F$-isocrystal obtained by restriciting $\V_{\mathrm{FL}}[1/p]$ to a point in $\bar\fS_p$.  
\begin{enumerate}
\item We have $N \Fil^q\D\subset \Fil^{q-1}\D$.
\item The weight filtration $W_\bullet \D$ of $\D$ with respect to $N$ is by weakly admissible sub-log $F$-isocrystals.
\end{enumerate}
\end{lemma}
\begin{proof}For part 1, the filtered flat vector bundle $\V_\mathrm{FL}$ is the $p$-adic completion of a filtered flat logarithmic vector bundle $(V,\nabla,\Fil^\bullet)$ on $\bar\scrS$.  It suffices to verify that the filtered flat vector bundle $(\hat V,\nabla,\Fil^\bullet)$ obtained from taking the $\scrD$-adic completion of $(V,\nabla,\Fil^\bullet)[1/p]$ has the property that, if $\scrD_0$ is a codimension 1 stratum of $\scrD$ and $N$ the flat lift of the residue along $\scrD_0$, then $N\Fil^q\subset \Fil^{q-1}$.  In the Shimura variety case, this is because $N$ is a flat section of the flat vector bundle associated to the adjoint representation, for which $\Fil^{-1}$ is everything. 

For part 2, for each $k$, $N^k:\gr^W_{-k}\D\to\gr_k^W\D(k)$ is an isomorphism of filtered $F$-isocrystals.  On the other hand, the polarization form gives an isomorphism $\gr^W_{-k}\D\cong (\gr_k^W\D)^\vee$.  Thus, by \Cref{lem: selfdual filtered F isocrystal} and the weak admissibility condition on $\D$, for the largest $k$ for which $\gr_k^W\D\neq 0$ we have that $\gr_{\pm k}^W\D$ are weakly admissible.  Thus, $W_{k-1}\D/W_{k}\D$ is weakly admissible, and the claim follows by induction.     
\end{proof}
\begin{proof}[Proof of \Cref{lem: existence of weight filtration Qp local system fiber}]
    
Given the above lemmas, we need only verify the first claim.  The $N$ operator on $\bar{\V}_{\cris}[1/p]|_x(K_0) $ is $\sum v_iN_i$, where the $N_i$ are the flat lifts of the residues associated to the divisors containing $x$ and $v_i$ is the valuation of the defining equation of the $i$th divisor in $\tilde x$.  By \ref{good intersections}, the weight filtration induced by $N$ is equal to $({W_\bullet}\bar\V_\cris|_{\fT^\LOG}[1/p])|_x (K_0)$.
\end{proof}

\begin{remark}\label{retraction 2}
    \Cref{N GT} is the second instance of a retraction onto the boundary.  In the setting of a general variation of Hodge structures, the flat lifts of the residues will not be Griffiths transverse. In other words, given an arbitrary semistable Galois representation, consider the weakly admissible module associated to it. The filtration induced by the monodromy operator $N$ is in general \emph{not} a filtration by weakly admissible submodules, i.e. the analogue of Lemma \ref{N GT} does not hold for arbitrary semistable local systems. 
\end{remark}


We also require a compatibility statement. Let $\tilde{x}$ and $x$ be as in  Lemma \ref{lem: existence of weight filtration Qp local system fiber}. There are two ways to obtain a graded $F$-crystal on the crystalline site of $x = \Spec \Fpbar$ from this data.  On the one hand, we have the pullback $(\gr_\bullet ^W \bar\V_\cris|_{\fT^\LOG_p})|_{x} $ of the graded $F$-crystal $\gr_\bullet ^W \bar\V_\cris|_{\fT^\LOG_p}$ given by \Cref{prop: weight filtration crystal}.  On the other hand, we have a filtration $W_\bullet{\VetZp}|_{\tilde{x}}[1/p]$ of the fiber of the $\Q_p$-local system as in \Cref{lem: existence of weight filtration Qp local system fiber} which we may intersect with the $\Z_p$-lattice to get $W_\bullet {\VetZp}|_{\tilde{x}}:={\VetZp}|_{\tilde{x}}\cap W_\bullet{\VetZp}|_{\tilde{x}}[1/p]$.  This yields a filtered  log  prismatic $F$-crystal on $(\Spf\cO_K)^\LOG$ where the latter is given the log structure associated to $\N \rightarrow \mathcal O_K$ sending $1 \mapsto \varpi$, 
hence an associated filtered log $F$-crystal $W_\bullet\D$ on $x^\LOG$ via the crystalline realization.  The associated graded $\gr^W_\bullet \D$ is a graded $F$-crystal on $x$.
\begin{proposition}\label{prop:compatibility of crystals}
    Setup as above. Then $\gr^W_\bullet \D$ and $(\gr_\bullet ^W \bar\V_\cris|_{\fT^\LOG_p})|_x$ are naturally isomorphic as graded $F$-crystals.  
\end{proposition}

\begin{proof}
    The morphism $f:\Spf \cO_K\to \fS$ coming from $\tilde x$ is compatible with the natural log structures, so we obtain a morphism $\hat g: (\Spf \mathcal O_K)^{\log} \rightarrow \bar{\mathfrak S}$ which reduces to a map   $g:x^\LOG\to \fT_p^\LOG$ on the special fiber. There is a natural functor from log Fontaine--Laffaille modules over $\bar{\mathfrak S}$ to log prismatic $F$-crystals over $\bar{\mathfrak S}$ constructed as follows. First there is a natural functor from log Fontaine--Laffaille modules over $\bar{\mathfrak S}$ to the category $ \textup{Loc}^{\textup{st}}_{\Z_p, [0, p -2]}(\bar{\mathfrak S} [1/p])$ of $\Z_p$-local systems over $\bar{\mathfrak S} [1/p]$ that are semistable in the sense of \cite[Definition 4.18]{Inoue}, with Hodge--Tate weights contained in $[0, p-2]$. 
    This latter category is in turn naturally equivalent to the category $
    \Vect^{\varphi}_{[0, p-2]} (\bar{\mathfrak S}_{\Prism})$ of log prismatic $F$-crystals over $\bar{\mathfrak S}$ with Hodge--Tate weights in the Fontaine--Laffaille range by an upcoming result of Inoue.\footnote{More precisely, building on \cite[Theorem A]{Inoue} and \cite[Theorem 6]{Koshikawa2},  Inoue proves that the \'etale realization functor induces a natural equivalence  
   $\Vect^{\textup{an}, \varphi} (\bar{\mathfrak S}_{\Prism}) \isom \textup{Loc}^{\textup{st}}_{\Z_p}(\bar{\mathfrak S} [1/p])$ between analytic log prismatic $F$-crystals on $\bar{\mathfrak S}$ (see Definition \ref{def:Guo_Reinecke_analytic} and \cite[Definition 4.3]{Inoue}) and semistable Kummer \'etale local systems on $\bar{\mathfrak S} [1/p]$. His proof extends techniques used in \cite{GuoReinecke} and relies on the log Beilinson fiber sequence \cite[Theorem 13.8]{Diao_logainf1}. In the Fontaine--Laffaille range, the analytic log prismatic $F$-crystals automatically extend to log prismatic $F$-crystals. 
    } In particular, the  log prismatic $F$-crystal associated to ${\VetZp}|_{\tilde{x}}$ is naturally the pullback of the log prismatic   $F$-crystal on $\bar\fS$ associated to $\V_\mathrm{FL}$ to $x$, hence the same is true for the crystalline realizations.  We therefore have two filtrations $W_\bullet \D$ and $g^*W_\bullet \bar\V_\cris|_{\fT^\LOG_p}$ of naturally identified log $F$-crystals.  Both are locally split (on the level of $\Z_p$-modules), hence saturated, and by \Cref{lem: existence of weight filtration Qp local system fiber} agree on the level of log $F$-isocrystals, hence they agree (as graded $F$-crystals).
\end{proof}

\subsection{Filtration by semistable local systems}




Suppose given $f:\Dstar\to \scrS^\an$.  We now work in the setting of $\Ann_n \subset \Dstar$. There are integral models and maps
\[\begin{tikzcd}
    \mathfrak{A}_n^\LOG\ar[d]\ar[rd,"g"]&\\
    \mathfrak{A}_n'^\LOG\ar[r,"g'"] &\overline{\fS}^\LOG
\end{tikzcd}\]with $\mathfrak{A}_n$ semistable and $\mathfrak{A}'_n$ as in \cite[Appendix]{p-adic-borel-Ag}, such that at the level of rigid generic fibers, the vertical map is the identity and both $g,g'$ induce $\Ann_n \rightarrow \bar\fS^{\rig}$. By Proposition \ref{prop: tube over bb stratum} the image of $\Dstar$ is contained in the tube $T_{\fT}$ over a Baily--Borel stratum $\fT$. We have the following crucial result, which follows from \cite[Proposition 3.46 (c)]{DLMS2}. 

\begin{proposition}\label{prop:sublocalsystemonannuli}
    There is a filtration $W_\bullet f^*\VetZp$ of $f^*\VetZp$ such that $W_\bullet f^*\VetZp|_{\Ann_n}$ is associated to the filtration $g_p^*W_\bullet\V_\cris|_{\fT_p^\LOG}[1/p]$ of the log $F$-isocrystal $g_p^*\V_\cris=g_p^*\V_\cris|_{\fT_p^\LOG}[1/p]$. This filtration is compatible with point-wise filtration constructed in Lemma \ref{lem: existence of weight filtration Qp local system fiber}. 
\end{proposition}

\begin{proof}
    Pick some classical point $\tilde{x} \in \Ann_n(K)$ which specializes to $x \in \mathfrak{A}_k \times \Fpbar$.  By \Cref{prop:compatibility of crystals} and the preceding discussion, we have a filtration $W_\bullet\VetZp|_{\tilde x}$ by Galois sub-representations, and it suffices to prove that this filtration is stable by the action of $\pi_1(\Ann_n)$ on ${\VetZp}|_{\tilde{x}}[1/p]$. For any particular $i$, by taking an appropriate wedge power of $\VetZp[1/p]$ and the log-isocrystal $\V_\cris[1/p]$, we may reduce to the case that $W_i\VetZp|_{\tilde x}$ is rank 1. We let $\mathbb{L}$ (resp. $\D$) be the wedge power of $\VetZp[1/p]$ (resp. $\V_\cris[1/p]$), and $\mathbb{E}\subset \D$ the rank 1 sub.  Note that $\mathbb{E}$ is in fact a $F$-isocrystal (not a log one) on $\fT^\LOG_p$.  After twisting, we may in fact assume $\mathbb{E}$ is unit-root. As the special fiber of $\mathfrak{A}'_n$ is a tree consisting of rational curves, we have that $g_p'^*\mathbb{E}$ is constant, and therefore that $g_p^*\mathbb{E}$ is constant.

From the log Fontaine--Laffaille module structure, $g^*\D$ naturally has the structure of a filtered log $F$-isocrystal. Consider the filtered $F$-isocrystal $(g_p^*\mathbb{E},\Fil^\bullet)$ on $\mathfrak{A}_n\times \Fpbar$ in the sense of \cite[B.37]{DLMS2} where $\Fil^\bullet$ is the filtration of the flat bundle $\mathbb{E}(\Ann_n)$ given by $\Fil^1=0$ and $\Fil^0=\mathbb{E}(\Ann_n)$.  The triviality of the $F$-isocrystal $g^*\mathbb{E}$ implies that this filtered $F$-isocrystal is associated to the constant rank 1 local system on $\Ann_k$. Consider the containment in \cite[Proposition 3.45 (c)]{DLMS2}. The sub-isocrystal $\mathbb{E} \subset \D$ induces the inclusion $\mathbb{E}(\Ann_k) \subset \D(\Ann_n)$ of flat vector bundles. In order to realize the rank 1 trivial local system as a sub-local system of $g^*\mathbb{L}$, it suffices to prove that the inclusion $\mathbb{E}(\Ann_n) \subset \D(\Ann_n)$ is strict with respect to the filtrations. This can be checked point-wise, and therefore follows from weak admissibility established in the previous section.
\end{proof}

\subsection{Proof of \Cref{thm:borel boundary}}

By Proposition \ref{prop: tube over bb stratum}, we have that $f$ factors through a single tube $T_{\scrT}$. By Proposition \ref{prop:sublocalsystemonannuli}, we have a filtration $W_{\bullet} f^{*}\VetZp$ of $f^*\VetZp$, whose associated graded is a graded crystalline local system, which we denote by $\gr_\bullet^Wf^*\VetZp$. As in Section \ref{sec: crystalline extension} $\gr_\bullet^Wf^*\VetZp$ extends to a local system on $\sD$. We now use Theorem \ref{prop:existence_of_crystal} and shrink $\sD$ to deduce pointwise constancy of the $F$-crystal $\D_{\cris}(\gr_\bullet^Wf^*\VetZp)=f^*\gr^W_\bullet \bar\V_\cris|_{\fT^\LOG}$, by \Cref{prop:compatibility of crystals}.

The argument now closely follows the case of the interior. Let $\Ann_n \subset \Dstar$ be the sub-annulus with outer radius 1 and inner radius $1/p^n$, and let $\fA_n$ be an integral model as in \cite[Appendix]{p-adic-borel-Ag}. 
It suffices to prove that $\Ann_n$ maps to a single residue disc of $\scrS^{\BB}$. Let $C$ be the special fiber of $\Ann_n$. We have that $C$ maps to $\scrT_{\F_v}$. By the above, $f^*\gr^W_\bullet \bar\V_\cris|_{\fT^\LOG}$ is point-wise constant, and therefore by \Cref{thm: constancy of F-crystal} is constant. The $F$-crystal $\gr^W_\bullet (\D)$ is pulled-back from the boundary stratum $\scrS_1$ of $\scrS^{\BB}$ assocaited to $\scrT$. The scheme $\scrS_1$ is a Shimura variety in its own right and the argument in Proposition \ref{prop: constantFLconstantmap} applies to show that the map $C \rightarrow \scrS_1$ is constant. We have therefore proved that $\Ann_n$ maps to a single residue disc of $\scrS^{\BB}$ and the theorem follows.\qed


\section{Higher-dimensional extension theorem}\label{sec: higher borel}
In this section, we deduce the higher-dimensional extension theorem from the theorem for $\Dstar$.

\subsection{Preliminaries}We first prove some general statements regarding mermorphic extension.  We will work in the following setting. Let $X$ be a quasi-projective algebraic variety over $K$ with a projective compactification $\overline{X}$ over $K$. We fix a closed embedding $\overline{X} \hookrightarrow \PP^n_K$ into projective $n$-space. This corresponds to the data of a very ample line bundle on $\overline{X}$ relative to $K$ that we denote by $\mathscr{L}$, along with $n+1$ generating global sections $s_0,\ldots, s_n \in H^0(\overline{X},\mathscr{L})$. 

\begin{definition}
    We say that the compactification $X \subset \overline{X}$ \emph{satisfies one-dimensional Borel extension over all finite extensions of $K$}, if for any finite field extension $F$ of $K$, every analytic map $g^\times : \Dstar_F \rightarrow (X \times_K F)^\an$ over $F$ admits an analytic extension $g : \sD_F \rightarrow (\overline{X} \times_K F)^\an$.
\end{definition}

\begin{definition}
    Let $W,Y$ be reduced rigid analytic spaces and $Z\subset Y$ a closed subspace.  We say a morphism $f:Y\backslash Z\to W$ \emph{extends meromorphically over $Z$} if the (metric) closure of the graph of $f$ in $Y\times W$ is an analytic subspace.  Equivalently, there is modification $g:Y'\to Y$ and a commutative diagram
    \[\begin{tikzcd}
    &Y'\ar[ld,"g",swap]\ar[rd]&\\
    Y&&W\\
    &Y\backslash Z\ar[lu,"i"]\ar[ru,"f"]&
    \end{tikzcd}\]
    where $i$ is the inclusion.
\end{definition}

The main result of this section is the following, which says that the one-dimensional Borel extension implies that analytic maps from poly-punctured disks $(\Dstar)^a \times \sD^b$ into $X$ extend meromoprhically.
\begin{proposition}\label{extension-to-codim-1}
    Suppose $X \subset \overline{X}$ satisfies one-dimensional Borel extension over all finite extensions of $K$. Then, given any finite field extension $F$ of $K$, and any analytic map $h^\times : (\Dstar_F)^a \times \sD_F^b \rightarrow X_F^\an$, there is a closed analytic subspace $Z \subset \sD_F^{a+b}$ of codimension at least $2$ contained in the complement of $(\Dstar)^a \times \sD^b$ such that $h^\times$ extends to an analytic map $h : \sD_F^{a+b}\setminus Z \rightarrow \overline{X}_F^\an$.  Moreover, $h$ extends meromorphically over $Z$.
\end{proposition}

We need two preparatory lemmas first. In \autoref{triviality-of-pic}, we show that there are no non-trivial line bundles on $(\Dstar)^a \times \sD^b$, and in \autoref{meromorphicity-from-finiteness}, we prove that an analytic function on $(\Dstar)^a \times \sD^b$, whose restriction to every one-dimensional punctured disk extends meromorphically to the disk $\sD$, must itself extend meromorphically to $\sD^{a+b}$.

\begin{lemma}\label{triviality-of-pic}
Let $\Ann_0 := \Spa(K\langle z^{\pm 1} \rangle, \cO_K\langle z^{\pm 1}\rangle)$ denote the thin annulus over $K$ at radius 1 centered at 0. Let $a, b, c$ be non-negative integers. Then every line bundle on $(\Dstar)^a \times \sD^b \times \Ann_0^c$ is trivial. In particular, every line bundle on $(\Dstar)^a \times \sD^b$ is trivial. 
\end{lemma}
\begin{proof}
    We induct on $a$ with the base case of $a=0$, being a consequence of the fact that the affinoid algebra $K\langle t_1,\ldots,t_b, z_1^{\pm 1},\ldots, z_c^{\pm 1} \rangle$ is a UFD (see \cite[Theorem 3.25]{vanderput-cohomology}).
    Suppose $a \geq 1$ and $\mathcal{E}$ is a line bundle on $(\Dstar)^a \times \sD^b \times \Ann_0^c$. By the inductive hypothesis, the restriction of $\mathcal{E}$ to $(\Dstar)^{a-1} \times \Ann_0 \times \sD^b \times \Ann_0^c$ is trivial. We may thus glue $\mathcal{E}$ with the trivial line bundle on $(\Dstar)^{a-1} \times (\PP^1_K)^\an\setminus \{|z| < 1\} \times \sD^b \times \Ann_0^c$ to obtain a line bundle $\hat{\mathcal{E}}$ on $(\Dstar)^{a-1} \times \mathbb{A}^{1,\an}_K \times \sD^b \times \Ann_0^c$ that extends $\mathcal{E}$. By \cite[Prop. 3.6]{sigloch-homotopy} (see also \cite[Corollary 5]{kerz-saito-tamme} for the smooth affinoid case), $\hat{\mathcal{E}}$ is the pullback of a line bundle on $(\Dstar)^{a-1} \times \sD^b \times \Ann_0^c$. However, by the inductive hypothesis every line bundle on $(\Dstar)^{a-1} \times \sD^b \times \Ann_0^c$ is trivial and hence so is $\hat{\mathcal{E}}$ and therefore so is $\mathcal{E}.$
\end{proof}

\begin{lemma}\label{meromorphicity-from-finiteness}
    Let $f(\underline{z},\underline{t})$ be an analytic function on $(\Dstar)^a \times \sD^b$, such that for each $1 \leq i \leq a$, and each finite extension $F$ of $K$, and every $F$-valued point $P' = (c_1,\ldots,\widehat{c_i},\ldots,c_a,\underline{\tau}) \in (\Dstar)^{a-1} \times \sD^b$, the specialization $f(c_1,\ldots,z_i,\ldots,c_a,\underline{\tau}) \in \cO(\Dstar_F)$ has only finitely many zeroes on $\Dstar_F.$ Then $f(\underline{z},\underline{t})$ extends to a meromorphic function on $\sD^{a+b}.$
\end{lemma}
\begin{proof}
    Fix $1 \leq i \leq a$. To simplify notations, we set $z' = (z_1,\ldots,\widehat{z_i},\ldots,z_n)$. The function $f$ admits a power series development \[ f(\underline{z},\underline{t}) = \sum_{m \in \Z} a_m(z',\underline{t}) z_i^{m},\] 
    for analytic functions $a_m(z',\underline{t}) \in \cO((\Dstar)^{a-1}\times \sD^b).$ By the $p$-adic big Picard theorem, since for each specialization $P'$, the function $f(z_i,P')$ has only finitely many zeros in $\Dstar_F$, we have that for each such $P'$, the function $f(z_i,P')$ cannot have an essential singularity at $z_i=0$. That is to say that $a_m(P') = 0$ for all $m$ sufficiently small. In other words, there is an $m_0 \in \Z$ such that $P' \in V(\{a_l(z',\underline{t}) : l \leq m_0\}).$
    This being true for every classical point $P'$, implies that $(\Dstar)^{a-1}\times \sD^b = \cup_{m \in \Z} V(\{a_l(z',\underline{t}) : l \leq m\})$. By the Baire category theorem, this implies that there is an $m(i) \in \Z$ such that for $l < m(i)$, $a_l(z',\underline{t}) = 0$. In particular, $z_i^{-m(i)} f(\underline{z},\underline{t})$, extends analytically across the locus $z_i = 0$ as well. The function $\prod_{1\leq j \leq a} z_j^{-m(j)}f(\underline{z},\underline{t})$ thus extends analytically to the complement inside $\sD^{a+b}$ of the codimension 2 subvariety $\cup_{1 \leq r < s \leq a} V(z_r,z_s).$ Hence by the non-archimedean Hartog extension theorem \cite{bartenwerfer-hartog}, defines an analytic function on $\sD^{a+b}.$ This completes the proof.
\end{proof}

\begin{proof}[Proof of \Cref{extension-to-codim-1}] 
    We may assume that $F = K$. Let $h^\times : (\Dstar)^a \times \sD^b \rightarrow X^\an$ be an analytic map. By \autoref{triviality-of-pic}, we may pick a trivialization $\phi : (h^\times)^*(\mathscr{L}^\an) \xrightarrow{\sim} \cO_{(\Dstar)^a \times \sD^b}$ of the pullback line bundle $(h^\times)^*(\mathscr{L}^\an)$.    
    We let $\underline{z} := (z_1,\ldots,z_a)$ denote the coordinates on $(\Dstar)^a$ and $\underline{t} := (t_1,\ldots,t_b)$ the coordinates on $\sD^b.$
    Let $f_r(\underline{z}, \underline{t}) := \phi(s_r) \in \cO((\Dstar)^a \times \sD^b)$. 
    On classical points $(\underline{z},\underline{t}) \in (\Dstar)^a \times \sD^b$, the map $h^\times : (\Dstar)^a \times \sD^b \rightarrow X^\an \hookrightarrow \PP^{n,\an}_K$ is given in projective coordinates by $[f_0(\underline{z},\underline{t}) : \ldots : f_n(\underline{z},\underline{t})].$ For each $i \in \{1,\ldots,a\}$, and any finite extension $F$ of $K$ and an $F$-valued point $P' = (c_1,\ldots,\widehat{c_i},\ldots, c_a, \underline{\tau}) \in (\Dstar)^{a-1}\times \sD^b$, the analytic map $h^\times_{P'} : \Dstar_F \rightarrow \PP^{n,\an}_F$ given by $z_i\mapsto [f_0(c_1,\ldots,z_i,\ldots, c_a,\underline{\tau}): \ldots : f_n(c_1,\ldots,z_i,\ldots, c_a,\underline{\tau})]$,  by hypothesis extends analytically to a map $h_{P'} : \sD_F \rightarrow \PP^{n,\an}_F.$ In particular, there exist analytic functions $g_r(z_i) \in F\langle z_i \rangle$ (depending on $P'$) with no common zeroes on $\sD_F$ such that for all classical points $z_i \in \Dstar_F$,  $[f_0(c_1,\ldots,z_i,\ldots, c_a,\underline{\tau}):\ldots: f_n(c_1,\ldots,z_i,\ldots, c_a,\underline{\tau})] = [g_0(z_i):\ldots : g_n(z_i)]$. This implies that for all $r, s$, $f_r(c_1,\ldots,z_i,\ldots, c_a,\underline{\tau}) g_s(z_i) = f_s(c_1,\ldots,z_i,\ldots, c_a,\underline{\tau})g_r(z_i)$ as analytic functions on $\Dstar_F$ (with coordinate $z_i$). Since the $\{f_r(c_1,\ldots,z_i,\ldots, c_a,\underline{\tau})\}_{0 \leq r \leq n}$ and the $\{g_r(z_i)\}_{0\leq r \leq n}$ have no common zeroes in $\Dstar_F$, this in particular implies that for each $0 \leq r \leq n$, the zero set $V(f_r(c_1,\ldots,z_i,\ldots, c_a,\underline{\tau}))$ equals $V(g_r(z_i))$ in $\Dstar_F$ and is thus \emph{finite}. We conclude from \autoref{meromorphicity-from-finiteness} that there are non-negative integers $\{m(j) : 1\leq j \leq a\}$, such that for each $0 \leq r \leq n$, $F_r(\underline{z},\underline{t}) := \prod_{1\leq j \leq a} z_j^{m(j)} f_r(\underline{z},\underline{t})$ extends analytically to $\sD^{a+b}$ and such that for each $1 \leq j \leq a$, there is an $r$ such that $z_j \nmid F_r(\underline{z},\underline{t})$ in $\cO(\sD^{a+b})$. The set of common zeroes $Z := V(\{F_r(\underline{z},\underline{t}) : 0 \leq r \leq n\})$ has codimension at least 2 in $\sD^{a+b}$, and $h^\times$ extends analytically to the map $h : \sD^{a+b}\setminus Z \rightarrow \PP_K^{n,\an},$ given by $(\underline{z},\underline{t}) \mapsto [F_0(\underline{z},\underline{t}):\ldots : F_n(\underline{z},
    \underline{t})].$

    For the final claim, let $V\subset \sD^{a+b}\times \PP^{n,\an}_K$ be the subspace cut out by $x_iF_j-x_jF_i$, where the $x_i$ are homogeneous coordinates on $\PP^n_K$.  This subspace is equal to the graph $\Gamma$ of $h$ when intersected with $U:=(\sD^{a+b}\backslash Z)\times \PP^{n,\an}_K$.  There is therefore an irreducible component $V_0$ of $V$ for which $\Gamma=U\cap V_0$ (\cite[Cor. 2.2.9]{conrad-irreducible-comp}), and since $U\cap V_0$ is metrically dense in $V_0$, it follows that $V_0$ is the closure of $\Gamma$.
\end{proof}

\begin{remark}
    We remark that the meromorphicity in \Cref{extension-to-codim-1} is automatic for an analytic morphism defined outside codimension 2, as in the complex analytic case: 

\begin{lemma}
    Let $Y$ be a smooth connected rigid analytic space and $Z\subset Y$ a closed subspace of codimension $\geq 2$.  Then any morphism $f:Y\backslash Z\to X^\an$ extends meromorphically over $Z$. 
\end{lemma}

\begin{proof}
    We may assume $X=\PP^n_K$. Since $Z$ has codimension at least 2 in $Y$, the pull-back $f^\ast(\cO(1))$ extends to a line bundle on $Y$ (using the correspondence between line bundles and Weil divisors and applying Remmert-Stein which allows us to extend Weil divisors outside the codimension $\geq$ 2 analytic subvariety $Z$.) 
    The question of whether $f$ extends meromorphically across $Z$ to all of $Y$ is local on $Y$. We may therefore assume that the pull-back $f^\ast \cO(1)$ is trivial. Then the morphism $f : Y\setminus Z \rightarrow \PP^{n,\an}_K$ is described in homogenous coordinates by $n+1$ analytic functions $y \mapsto [f_0(y):\ldots : f_n(y)]$ where $f_i(y) \in \cO(Y\setminus Z)$ do not have any common zeroes in $Y\setminus Z$. By the non-archimedean Hartog extension theorem, $f_i(y)$ uniquely extend to analytic functions $F_i \in \cO(Y)$. The subspace in $Y \times \PP^{n,\an}_K$ cut out by $x_iF_j - x_jF_i$ contains as an irreducible component the closure inside $Y \times \PP^{n,\an}_K$ of the graph of $f : Y \setminus Z \rightarrow \PP^{n,\an}_K$.   
\end{proof}
\end{remark}

\subsection{Higher-codimension extension for Shimura varieties and period images}
\def\Griff{\operatorname{Griff}}
\def\cU{\mathcal{U}}
\def\Proj{\operatorname{Proj}}
\def\Can{{\C\mbox{-}\an}}
Let $X$ denote either $\ShimK$ or a geometric period image $Y$ as described in $\S1$.  For the reader's convenience we summarize here the structures that will be relevant for the proof of \Cref{thm:introextension}.
\begin{itemize}
\item A $\Z_p$-local system $\VetZp$ on $X$.
\item A filtered vector bundle $V_{\mathrm{dR}}:=(\mathcal{V},F^\bullet )$ on $X$.
\item A normal compactification $X^\BB$ of $X$ for which the $k$th (for some $k$) power of the Griffiths bundle
\[\Griff(V_{\mathrm{dR}}):=\bigotimes_p\det F^p\]
extends to an ample line bundle $L$.  This is \cite{bailyborel} in the Shimura variety case and \cite[Thm 1.2]{bfmt-bailyborel} in the period image case.
\item A log smooth proper $(X',D_{X'})$ with a modification $\pi_{X'\backslash D_{X'}}:X'\backslash D_{X'}\to X$, such that the pullback $U_{\mathrm{dR}}:=(\mathcal{U},\nabla,F^\bullet)$ of $(\mathcal{V},F^\bullet )$ admits a flat connection with respect to which the filtration is Griffiths transverse and extends to a logarithmic flat vector bundle $\bar U_{\mathrm{dR}}:=(\bar{ \mathcal{U}},\nabla,F^\bullet\bar{\mathcal{U}})$ such that the eigenvalues of the residues are contained in $[0,1)\cap \Q$.  Moreover, $\pi_{X'\backslash D_{X'}}:X'\backslash D_{X'}\to X$ extends to a morphism $\pi:(X',D_{X'})\to (X^\BB,X^\BB\backslash X)$ for which
\[\pi^*L\cong \Griff(\bar U_\mathrm{dR})^k.\]

Here, $X'$ is taken to be the largest stratified resolution $S^m$ in the period image case, and $S$ itself in the Shimura variety case.  Again, this follows from \cite{bailyborel} in the Shimura variety case and \cite[Thm 1.2]{bfmt-bailyborel} in the period image case. 
\item We have 
\begin{equation}\label{pullback isom}\bar U_{\mathrm{dR}}^{\an}\cong \Ddrlog(\pi_{X'\backslash D_{X'}}^*\VetZp) \end{equation}
via the $p$-adic Riemann--Hilbert correspondence of \cite[Thm 1.7]{DLLZ}.  In the period image case the isomorphism is given over $P'$ by the second part of \cite[Thm 1.1]{DLLZ}, and in the Shimura variety case by \cite[Thm 1.5]{DLLZ}.
\end{itemize}
\begin{proposition}\label{extension codim 2}  Let $(M,D_M)$ be a log smooth rigid-analytic variety and $f:M\backslash D_M\to X^\an$ a morphism such that $M\backslash D_M\to X^{\BB,\an}$ extends meromorphically over $D_M$.  Then it extends regularly over $D_M$.
\end{proposition}
\begin{proof}By meromorphicity and embedded resolution of singularities \cite[Thms 1.1.9, 1.1.13]{temkinembedded}, there is a log smooth $(M',D_{M'})$, a modification $g:(M',D_{M'})\to (M,D_M)$, and a diagram
\[
\begin{tikzcd}
    &(M',D_{M'})\ar[dd,"g"]\ar[rr,"h"]&&(X',D_{X'})\ar[dd,"\pi"]\\
    M'\backslash D_{M'}\ar[dd,"g_{M'\backslash D_{M'}}"]\ar[ru]&&\\
    &(M,D_M)&&(X^\BB,X^\BB\backslash X)\\
    M\backslash D_M\ar[ru]\ar[rr,"f"]&&X\ar[ru]&
\end{tikzcd}
\]
Since $(\pi h)_{M'\backslash D_{M'}}^*\V\cong (fg_{M'\backslash D_{M'}})^*\VetZp$, we have 
\[g^*\Ddrlog(f^*\VetZp)\cong h^*\Ddrlog(\pi_{X'\backslash D_{X'}}^*\VetZp) \]
and likewise for the Griffiths bundles.  Since $\Griff(\Ddrlog(\pi_{S'\backslash D_{S'}}^*\VetZp))^k$ descends amply to $X^\BB$, it follows that $\pi h$ factors through $g$, by the rigidity lemma (see e.g. \cite[Lemma 1.15]{debarre}).
\end{proof}

\subsection{Proof of \Cref{thm:introextension}}

    Combine \Cref{thm:borel boundary}, \Cref{extension-to-codim-1}, and \Cref{extension codim 2} in the case of Shimura varieties, and \Cref{disk period}, \Cref{extension-to-codim-1}, and \Cref{extension codim 2} for period images.\qed
\begin{remark}
    The main obstruction to using our strategy in the bad reduction case for period images is the lack of a ``retraction onto the boundary'' in the sense of \Cref{retraction 1} and \Cref{retraction 2}.
\end{remark}

\bibliographystyle{alpha}
\bibliography{main}
\end{document}